\documentclass[10pt]{article}
\usepackage{amssymb,amsmath}
\usepackage{amsthm}
\usepackage{graphicx}
\usepackage{eucal}
\usepackage{color}
\newtheorem{theorem}{Theorem}
\newtheorem{corollary}{Corollary}
\def\vec#1{{\boldsymbol #1}}
\newcommand{\x}{{\boldsymbol{x}}}
\newcommand{\R}{\mathbb{R}}
\newcommand{\der}{\mathrm{d}}
\newcommand{\lie}{\CMcal{L}}
\newcommand{\strang}{\CMcal{S}}
\newcommand{\Or}{\CMcal{O}}
\newcommand{\step}{\Delta t}
\newcommand{\new}{\mathrm{new}}
\newcommand{\old}{\mathrm{old}}
\newcommand{\splt}{\mathrm{split}}
\newcommand{\X}{\mathbf{X}}
\newcommand{\HO}{\mathrm{HO}}

\newcommand{\RK}{\mathrm{RK}}
\newcommand{\OS}{\mathcal{S}}
\newcommand{\Gpara}{\mathcal{G}}
\newcommand{\Fpara}{\mathcal{F}}
\newcommand{\Exsol}{\CMcal{T}}
\newcommand{\sumbeta}{\CMcal{B}}
\def\ds{\displaystyle}

\title{High order schemes
based on operator splitting and deferred
corrections for stiff time dependent PDE's}

\author{
Max Duarte\footnotemark[1]\ \footnotemark[2]
\and
Matthew Emmett\footnotemark[1]
}

\begin{document}

\maketitle

\renewcommand{\thefootnote}{\fnsymbol{footnote}}

\footnotetext[1]{
Center for Computational Sciences and Engineering,
Lawrence Berkeley National Laboratory,
1 Cyclotron Rd. MS 50A-1148,
Berkeley, CA 94720, USA
({\tt MWEmmett@lbl.gov}).
}

\footnotetext[2]{
{\it Present address:} 
CD-adapco, 200 Shepherds Bush Road, London W6 7NL, UK
({\tt max.duarte@cd-adapco.com}).
}

\renewcommand{\thefootnote}{\arabic{footnote}}

\begin{abstract}
We consider quadrature formulas of high order in time
based on Radau--type, $L$--stable
implicit Runge--Kutta schemes
to solve time dependent stiff PDEs.
Instead of solving a large nonlinear system of equations,
we develop a method that performs iterative deferred
corrections to compute the solution at the collocation nodes of the
quadrature formulas.
The numerical stability is guaranteed by a dedicated
operator splitting technique that efficiently handles the
stiffness of the PDEs and provides initial and intermediate
solutions to the iterative scheme.
In this way the low order approximations computed by
a tailored splitting solver of low algorithmic complexity
are iteratively corrected to obtain a high order solution
based on a quadrature formula.
The mathematical analysis of the numerical errors and
local order of the method is carried out in a finite dimensional framework
for a general semi--discrete problem, and a time--stepping strategy
is conceived to control numerical errors related to the time integration.
Numerical evidence confirms the theoretical findings
and assesses the performance of the method in the case of a
stiff reaction--diffusion equation.
\end{abstract}



\pagestyle{myheadings}
\thispagestyle{plain}
\markboth{DUARTE, EMMETT}
{DEFERRED CORRECTION OPERATOR SPLITTING SCHEMES}

\section{Introduction}\label{sec:intro}
Operator splitting techniques were originally introduced
with the main objective of saving computational costs compared to fully coupled techniques.
A complex and potentially large problem could be then
split into smaller parts or subproblems of different nature
with a significant reduction of the algorithmic complexity
and computational requirements.
Operator splitting techniques have been used over the past years
to carry out numerical simulations in several domains,
from biomedical models
to combustion or air pollution modeling applications.
Moreover, these methods continue to be widely used due to their
simplicity of implementation and their high degree of liberty
in terms of the choice of dedicated numerical solvers for the split subproblems.
In particular, they are suitable for stiff problems
for which robust and stable methods that properly handle and damp out
fast transients inherent to the different processes must be used.
In most applications, first and second order splitting schemes
are implemented, for which a general mathematical background
is available (see, e.g., \cite{MR2221614} for ODEs and
\cite{MR2002152} for PDEs) that
characterizes splitting errors originating from
the separate evolution of the split subproblems.
Even though higher order splitting schemes have been
also developed,
more sophisticated numerical implementations are required
and their applicability is currently limited to specific
linear or nonstiff problems
(see, e.g., \cite{Descombe01,Castella09,Hansen09,Blanes13}
and discussions therein).

In the past decades high order and stable
implicit Runge--Kutta schemes have been developed and widely
investigated to solve stiff problems modeled by ODEs
(see \cite{Hairer96} Sect. IV and references therein).
These advances can be naturally exploited to solve
stiff semi--discrete problems originating from
PDEs discretized in space, as considered, for instance,
in \cite{Bijl2002,Carpenter05,Bijl2005} to simulate compressible flows.
Here we are, in particular, interested in
a class of fully implicit Runge--Kutta schemes,
built upon collocation methods that use
quadrature formulas to numerically approximate
integrals \cite{Guillon69,Wright71}.
However, the high performance of implicit Runge--Kutta methods
for stiff ODEs
is adversely affected when applied to
large systems of nonlinear equations
arising in the numerical integration of semi--discrete PDEs.
Significant effort is thus required to
achieve numerical implementations that
solve the corresponding algebraic problems
at reasonable computational expenses.
As an alternative to building such a high order implicit solver
we consider
low order operator splitting schemes,
specifically conceived to solve stiff PDEs,
embedded in a classical iterative deferred correction scheme
(see, e.g., \cite{Skeel82})
to approximate the quadrature formulas of
an $s$--stage implicit Runge--Kutta scheme.

In this work, the high order quadrature formulas over a time step $\step$,
corresponding to an $s$--stage implicit Runge--Kutta scheme,
are evaluated using the numerical approximations
computed by a splitting solver at the $s$ intermediate collocation nodes.
Such a dedicated splitting solver for stiff PDEs can be built based on
the approach introduced in \cite{Duarte11_SISC}.
This approach relies on the use of one--step and high order
methods with appropriate stability properties and
time--stepping features for the numerical integration of the split subproblems.
The splitting time step can be therefore defined
independently of standard stability constraints
associated with mesh size or stiff source time scales.
All the numerical integration within a given $\step$
is thus performed by the splitting solver with no stability restrictions,
while the fully coupled system without splitting is evaluated
for the quadrature formulas.
Starting from a low order splitting approximation, the
stage solutions are then iteratively corrected
to obtain a high order quadrature approximation.
This is done in the same spirit of other iterative schemes
that correct low order approximations to obtain
higher order numerical solutions of ODEs or PDEs,
like, for instance, the parareal algorithm \cite{Lions01}
or SDC (Spectral Deferred Correction) schemes \cite{DuttSDC2000}.

This paper is organized as follows.
Time operator splitting as well as implicit Runge--Kutta
schemes are briefly introduced in Section~\ref{sec:math_back}.
The new deferred correction splitting algorithm is introduced
in Section~\ref{sec:dcs_scheme}.
A mathematical analysis of the numerical errors is conducted
in Section~\ref{sec:analysis}, and a time step selection strategy
with error control is subsequently derived in Section~\ref{sec:error_control}.
Spatial discretization errors are briefly discussed in
Section~\ref{sec:space}.  Finally,  theoretical findings are
numerically investigated
in Section~\ref{sec:num_res} for a stiff reaction--diffusion problem.

\section{Mathematical background}\label{sec:math_back}
We consider a parabolic, time dependent PDE given by
\begin{equation}\label{eq:gen_prob}
\left.
\begin{array}{ll}
\partial_t u = F(\partial^2_{\x} u, \partial_{\x} u, u), &
t > t_0,\, \x \in \R^d,\\
u(t_0,\x) = u_0(\x),&
t = t_0,\, \x \in \R^d,
\end{array}
\right\}
\end{equation}
where $u:\R\times \R^d \to \R^m$ and
$F:\R^m \to \R^m$.
For many physically inspired systems, the right hand side $F$ can be split into components according to
\begin{equation}\label{eq:def_F}
F(\partial^2_{\x} u, \partial_{\x} u, u) =
F(u) = F_1(u) + F_2(u) + \ldots,
\end{equation}
where the $F_i(u) $, $i=1,\ldots$, represent different
physical processes.
For instance, a scalar nonlinear reaction--diffusion equation
with $u:\R\times \R^d \to \R$
would be split into $F_1(u)=- \partial_{\x} \cdot (D(u) \partial_{\x} u)$
and $F_2(u)= f(u)$ for some diffusion coefficient
$D:\R \to \R$, and nonlinear function $f:\R \to \R$.
Within the general framework of
(\ref{eq:gen_prob}), we consider the semi--discrete problem
\begin{equation}\label{eq:gen_prob_disc}
\left.
\begin{array}{ll}
\der_t \vec u = \vec F(\vec u), &
t > t_0,\\[1ex]
\vec u(0) = \vec u_0,&
t = t_0,
\end{array}
\right\}
\end{equation}
corresponding to problem (\ref{eq:gen_prob})
discretized on a grid $\X$ of size $N$;
$\vec u,\, \vec u_0 \in \R^{m\cdot N}$ and
$\vec F:\R^{m\cdot N} \to \R^{m\cdot N}$.
Analogous to (\ref{eq:def_F}), we split $\vec F(\vec u)$ into components
$\vec F_i(\vec u)$.

\subsection{Time operator splitting techniques}
Assuming for simplicity
that $\vec F(\vec u) = \vec F_1(\vec u) + \vec F_2(\vec u)$,
we denote the solution of the subproblem
\begin{equation*}\label{eq:prob_F1}
\left.
\begin{array}{ll}
\der_t \vec v = \vec F_1(\vec v), &
t > t_0,\\
\vec v(0) = \vec v_0,&
t = t_0,
\end{array}
\right\}
\end{equation*}
by $X^t \vec v_0$, and the solution of the subproblem
\begin{equation*}\label{eq:prob_F2}
\left.
\begin{array}{ll}
\der_t \vec w = \vec F_2(\vec w), &
t > t_0,\\
\vec w(0) = \vec w_0,&
t = t_0.
\end{array}
\right\}
\end{equation*}
by $Y^t \vec w_0$.
The Lie (or Lie--Trotter \cite{Trotter59}) splitting approximations to
the solution of problem (\ref{eq:gen_prob_disc})
are then given by
\begin{equation}\label{eq:lie}
\lie_1^t \vec u_0 = X^t Y^t \vec u_0,\qquad
\lie_2^t \vec u_0 = Y^t X^t \vec u_0.
\end{equation}
Lie approximations are of first order in time;
second order can be achieved by using symmetric
Strang (or Marchuk \cite{Marchuk68}) formulas \cite{Strang68} to obtain
\begin{equation}\label{eq:strang}
\strang_1^t \vec u_0 = X^{t/2} Y^t X^{t/2} \vec u_0,\qquad
\strang_2^t \vec u_0 =  Y^{t/2} X^t Y^{t/2} \vec u_0.
\end{equation}
Similar constructions follow for more than two
subproblems in (\ref{eq:def_F}).
Let us denote by $\OS^{\step} \vec u_0$ any of these four splitting approximations,
where $\step$ represents the {\it splitting time step} (i.e., the overall time--marching algorithm would compute $\vec u_{n+1} = \OS^{\step} \vec u_n$).
In this work we consider splitting approximations
built in practice under the precepts established in \cite{Duarte11_SISC}.
In particular, dedicated solvers with time--stepping features are implemented
to separately integrate each split subproblem during successive splitting time steps
such that the numerical stability of the method is always guaranteed.

\subsection{Time implicit Runge--Kutta schemes}
Let us now consider
an implicit $s$--stage Runge--Kutta scheme
to discretize (\ref{eq:gen_prob_disc}) in time.
Given a time step $\step$, the solution $\vec u(t_0+\step)$
is approximated by $\vec u_{1}^{\RK}$, computed as
\begin{align}
&\vec u_i = \vec u_0 + \step \ds \sum_{j=1}^{s} a_{ij} \vec F \left(t_0+c_j \step, \vec u_j\right),
\qquad i=1,\ldots,s; \label{eq2:runge_kutta_1}\\
&\vec u_1^{\RK}= \vec u_0+ \step \ds \sum_{j=1}^s b_j \vec F \left(t_0+c_j \step, \vec u_j\right).
\label{eq2:runge_kutta_2}
\end{align}
The arrays $\vec b$, $\vec c \in \R^s$ gather the various coefficients
$\vec b=(b_1, \ldots, b_s)^T$ and $\vec c=(c_1,\ldots, c_s)^T$, and
$\vec A \in \CMcal{M}_s(\R)$ such that
 $\vec A=(a_{ij})_{1 \leq i,j \leq s}$.
These coefficients
define the stability properties and the order
conditions of the method,
and are usually arranged in a Butcher tableau according to
\begin{equation*}
\begin{tabular}{c|c }
$\vec c$ & $\vec A$ \\
\hline \\[-2.5ex]
& $\vec b^T$
\end{tabular}.
\end{equation*}
Recall that, if all the elements of the matrix of coefficients $\vec A$
are nonzero, then the scheme is said to be a
{\it fully IRK scheme} \cite{Hairer96}.
Moreover, if
\begin{equation}\label{eq:stiffly_acc}
a_{sj} = b_j, \quad j = 1,\ldots,s,
\end{equation}
then the last stage $\vec u_s$ corresponds to the solution
$\vec u_1^{\RK}$ according to (\ref{eq2:runge_kutta_1})--(\ref{eq2:runge_kutta_2}).

Methods satisfying (\ref{eq:stiffly_acc})
are called {\it stiffly accurate} \cite{Prothero74}
and are particularly appropriate for the solution of
(stiff) singular perturbation problems and for
differential--algebraic equations \cite{Hairer96}.
Fully IRK schemes
with a number of stages inferior to its approximation order
can be built
based on collocation methods
\cite{Guillon69,Wright71},
together with the simplified order conditions
introduced by Butcher \cite{MR0159424}.
In this case, the coefficients
$\vec b$ 
correspond to the
quadrature formula
of order $p$ such that
$\int_{0}^{1} \pi(\tau)\, \der \tau
=
\sum_{j=1}^s b_j \pi (c_j)$
for polynomials $\pi(\tau)$ of degree $\leq p-1$.  Moreover, the coefficients in $\vec c$ and $\vec A$, together with conditions for the stage order $q$,
imply that
at every stage $i=1,\ldots,s$ the quadrature formula
$\int_{0}^{c_i} \pi(\tau)\, \der \tau
=
\sum_{j=1}^s a_{ij} \pi (c_j)$
holds for polynomials $\pi(\tau)$ of degree $\leq q-1$.
Depending on the quadrature formula considered,
for instance,
Gauss, Radau or Lobatto, different families of
implicit Runge--Kutta methods can be constructed
(for more details, see \cite{Hairer96} Sect. IV.5).

In this work we consider the family of RadauIIA
quadrature formulas 
introduced by Ehle \cite{Ehle69},
based on \cite{Butcher64},
that consider
Radau quadrature formulas \cite{RadauC}
such that $p=2s-1$ and $q=s$.
These are
$A$-- and $L$--stable schemes that are
stiffly accurate methods according to
(\ref{eq:stiffly_acc}).
Note that, even though Gauss methods attain a maximum order
of $p=2s$ \cite{MR0159424,Ehle68},
they are neither
stiffly accurate nor $L$--stable schemes,
which are both important properties for stiff problems.
Approximations of less order are obtained
with Lobatto methods satisfying
$p=2s-2$ \cite{MR0159424,Ehle68,Chipman71,Axelsson72}.
In particular the collocation
methods with $p=2s-2$ and $q=s$,
known as the LobattoIIIA methods, yield
stiffly accurate schemes, but these are only $A$--stable.
As such, we hereafter focus on the
RadauIIA quadrature formula of order 5 given
by (\cite{Hairer96} Table IV.5.6)
\begin{equation}\label{eq:radau5}
\begin{tabular}{c|ccc}
$\ds \frac{4-\sqrt{6}}{10}$&$\ds\frac{88-7\sqrt{6}}{360}$&$\ds\frac{296-169\sqrt{6}}{1800}$&$\ds\frac{-2+3\sqrt{6}}{225}$\\[1.5ex]
$\ds \frac{4+\sqrt{6}}{10}$&$\ds\frac{296+169\sqrt{6}}{1800}$&$\ds\frac{88+7\sqrt{6}}{360}$&$\ds\frac{-2-3\sqrt{6}}{225}$\\[1.5ex]
$1$&$\ds\frac{16-\sqrt{6}}{36}$&$\ds\frac{16+\sqrt{6}}{36}$&$\ds\frac{1}{9}$\\[1.25ex]
\hline\\[-2ex]
&$\ds\frac{16-\sqrt{6}}{36}$&$\ds\frac{16+\sqrt{6}}{36}$&$\ds\frac{1}{9}$
\end{tabular}.
\end{equation}

\section{Iterative construction of high order schemes}\label{sec:dcs_scheme}
With this background, we describe in what follows
how a high order quadrature approximation can be computed
through an iterative scheme.
Following the IRK scheme (\ref{eq2:runge_kutta_1}), we define
$t_i=t_0+c_i\step$, $ i=1,\ldots,s$ and
introduce the set
$\vec U=(\vec u_i)_{i=1,\ldots,s}$, which is
comprised of the
approximations to the
solution
$\vec u(t_i)$
of (\ref{eq:gen_prob_disc}) at the intermediate times.
Since $c_s=1$ for stiffly accurate schemes,
and in particular for RadauIIA quadrature formulas,
$\vec u_s$ stands for the approximation to
$\vec u(t_0+\step)$, which was denoted as $\vec u_1^{\RK}$
in (\ref{eq2:runge_kutta_2}).
To simplify the discussion that follows, we define
\begin{equation}\label{eq:def_I}
I_{t_0}^{t_i} (\vec U) :=
\step \ds \sum_{j=1}^{s} a_{ij} \vec F\left(t_0+c_j \step, \vec u_j\right),
\end{equation}
so that the IRK scheme (\ref{eq2:runge_kutta_1}) can be recast as
\begin{equation}\label{eq2:runge_kutta_recast}
\vec u_i = \vec u_0 + I_{t_0}^{t_i} (\vec U), \qquad
i=1,\ldots,s,
\end{equation}
or equivalently as
\begin{equation}\label{eq2:runge_kutta_recast2}
\vec u_i = \vec u_{i-1} + I_{t_{i-1}}^{t_i} (\vec U), \qquad
i=1,\ldots,s,
\end{equation}
where
$I_{t_{i-1}}^{t_i} (\vec U ) :=
I_{t_0}^{t_i} (\vec U )- I_{t_0}^{t_{i-1}}(\vec U)$.
Notice that
$\int_{c_i}^{c_l} \pi(\tau)\, \der \tau
=
\sum_{j=1}^s ( a_{lj}-a_{ij}) \pi (c_j)$
still holds
for
$l>i$ and
polynomials $\pi(\tau)$ of degree $\leq q-1$; and therefore
$I_{t_{i-1}}^{t_i} (\vec U)$
retains the stage order $q$.

To compute $\vec U$, that is, the stage values of
the IRK scheme (\ref{eq2:runge_kutta_1}),
and hence the Runge--Kutta approximation (\ref{eq2:runge_kutta_2}),
we need to solve
a nonlinear system of equations
of size $m\times N\times s$
given by
(\ref{eq2:runge_kutta_1}) or (\ref{eq2:runge_kutta_recast}).
One of the most common ways to do this consists of
implementing Newton's method
(see, e.g., \cite{Hairer96} Sect. IV.8).
In this work, however, we approximate $\vec U$
by an iterative deferred correction technique.

\subsection{Deferred correction splitting scheme}\label{sec:DCS}

Given a provisional set of solutions denoted by $\widetilde{\vec U}=(\widetilde{\vec u}_i)_{i=1,\ldots,s}$,
we can compute an approximation
$\widehat{\vec U}=(\widehat{\vec u}_i)_{i=1,\ldots,s}$ to $\vec U$, following
(\ref{eq2:runge_kutta_recast2}), according to
\begin{equation}\label{eq2:def_uhat}
\widehat{\vec u}_i = \widetilde{\vec u}_{i-1} + I_{t_{i-1}}^{t_i} (\widetilde{\vec U}), \qquad
i=1,\ldots,s,
\end{equation}
with $\widetilde{\vec u}_0=\vec u_0$.
This approximation is then iteratively corrected
by computing
\begin{equation}\label{eq2:def_utildek}
\widetilde{\vec u}^{k+1}_i = \widehat{\vec u}^k_i + \vec \delta^k_i, \qquad
i=1,\ldots,s,
\end{equation}
which defines a new set of provisional solutions $\widetilde{\vec U}^{k+1}$ corresponding
to iteration $k+1$,
and a new approximate solution $\widehat{\vec U}^{k+1}$ according to (\ref{eq2:def_uhat}).
When the corrections $(\vec \delta_i)_{i=1,\ldots,s}$
become negligible, we expect that $\widetilde{\vec U}$
approaches the quadrature formulas used in the IRK scheme
(\ref{eq2:runge_kutta_1}) (or (\ref{eq2:runge_kutta_recast})),
with $\widetilde{\vec u}_s$ approximating 
$\vec u_1^{\RK}$
in (\ref{eq2:runge_kutta_2}).

Taking into account that system  (\ref{eq:gen_prob})
(or (\ref{eq:gen_prob_disc}))
is stiff,
we follow the approach of \cite{Duarte11_SISC}
and consider an operator splitting technique with dedicated
time integration schemes for each subproblem
to handle stiffness and
guarantee numerical stability.
We thus
embed such an operator splitting technique
((\ref{eq:lie}) or (\ref{eq:strang}))
into the deferred correction scheme
(\ref{eq2:def_uhat})--(\ref{eq2:def_utildek}),
yielding the deferred correction splitting (DC--S) algorithm.
The first approximation $\widetilde{\vec U}^0$
is obtained directly by recursively applying an operator splitting
scheme as follows:
\begin{equation}\label{eq:ini_split}
\widetilde{\vec u}^0_1=
\OS^{c_1\step} \vec u_0, \qquad
\widetilde{\vec u}^0_i=
\OS^{(c_i-c_{i-1})\step} \widetilde{\vec u}^0_{i-1}, \quad
i=2,\ldots,s.
\end{equation}
We then define the corrections in (\ref{eq2:def_utildek}) as
\begin{equation*}\label{eq:def_delta}
\vec \delta^k_i = \OS^{(c_i-c_{i-1})\step} \widetilde{\vec u}^{k+1}_{i-1}
- \OS^{(c_i-c_{i-1})\step} \widetilde{\vec u}^{k}_{i-1},\qquad
i=2,\ldots,s.
\end{equation*}
Noticing that $\vec \delta^k_1=\vec 0$, (\ref{eq2:def_utildek})
becomes
\begin{equation}\label{eq2:def_utildeksplit}
\widetilde{\vec u}^{k+1}_1 = \widehat{\vec u}^k_1,\qquad
\widetilde{\vec u}^{k+1}_i = \widehat{\vec u}^k_i +
\OS^{(c_i-c_{i-1})\step} \widetilde{\vec u}^{k+1}_{i-1}
- \OS^{(c_i-c_{i-1})\step} \widetilde{\vec u}^{k}_{i-1}, \quad
i=2,\ldots,s.
\end{equation}

The numerical time integration of problem (\ref{eq:gen_prob_disc})
is thus performed using an operator
splitting technique throughout the time step $\step$.
These results are subsequently used to approximate the solutions $(\widetilde{\vec u}^k_j)_{j=1,\ldots,s}$ at the collocation nodes
in the quadrature formula (\ref{eq2:def_uhat}), corresponding
in our case to a RadauIIA quadrature formula.
The coefficients of matrix $\vec A$ are
embedded in the definition of $I_{t_{i-1}}^{t_i} (\widetilde{\vec U})$ in (\ref{eq2:def_uhat})
following (\ref{eq:def_I}),
while coefficients $\vec c$ define the length of the time substeps
within a given $\step$.
Note that the fully coupled
$\vec F(\widetilde{\vec u}^k_j)$, $j=1,\ldots,s$,
is evaluated in $I_{t_{i-1}}^{t_i} (\widetilde{\vec U}^k)$ (eq. (\ref{eq2:def_uhat})) to compute
$\widehat{\vec u}^k_i$, $i=1,\ldots,s$. 
Denoting as $\widehat{p}$ the order of the operator
splitting scheme, we will demonstrate in the following section
that each iteration increases the order of
the initial numerical approximation (\ref{eq:ini_split}) by one,
that is, the local error of the
method (\ref{eq2:def_utildeksplit})
behaves like $\Or(\step^{\widehat{p}+k+1})$,
potentially up to the order $p$ of the quadrature formula
and at least up to its stage order plus one: $q+1$.
As previously said, the numerical stability of the time integration process
is guaranteed by the use of a dedicated operator splitting solver \cite{Duarte11_SISC}, 
whereas the quadrature formulas correspond to the $L$--stable RadauIIA--IRK scheme.
Consequently, $\step$ is not subject to any stability
constraint and a time--stepping criterion to select the time step size based on some error estimate should be introduced.
In particular if the splitting solver considers only time--implicit $L$--stable methods, 
the overall DC--S scheme will also be $L$--stable. 
However, one might be interested in using less computationally expensive 
explicit methods within the splitting solver.

The relation of the DC--S scheme
(\ref{eq:ini_split})--(\ref{eq2:def_utildeksplit})
with other iterative methods in the literature,
namely, the parareal algorithm \cite{Lions01}
and SDC schemes \cite{DuttSDC2000}, is briefly discussed
in Appendices \ref{sec:parareal} and \ref{sec:SDC}.

\section{Analysis of numerical errors}\label{sec:analysis}
We investigate the behavior of the numerical error associated with the
DC--S scheme (\ref{eq:ini_split})--(\ref{eq2:def_utildeksplit}),
considering problem (\ref{eq:gen_prob_disc})
in a finite dimensional setting.
Let $X$ be a Banach space
with norm $\|\cdot\|$
and
$\vec F$ an unbounded nonlinear operator form $D(\vec F)\subset X$ to $X$.
We assume that $\vec u(t)$, which is the solution of (\ref{eq:gen_prob_disc}),
also belongs to $X$
and that the same follows for the solutions of the split subproblems.
Assuming a Lipschitz condition for $\vec F(\vec u)$ given by
\begin{equation}\label{eq:lip_F}
 \left\|\vec F(\vec u)-\vec F(\vec v)\right\| \leq \kappa \|\vec u-\vec v\|,
\end{equation}
we have, given (\ref{eq:def_I}), that
\begin{equation}\label{eq:diff_I}
 \left\|I_{t_{i-1}}^{t_i}(\vec U)-I_{t_{i-1}}^{t_i}(\vec V)\right\|
 \leq C_1 \step  \ds \sum_{j=1}^{s} \|\vec u_j-\vec v_j\|,\qquad
i=1,\ldots,s.
\end{equation}
Furthermore, considering (\ref{eq:def_I}) with the exact solution
$\vec u(t)$ gives
\begin{equation}\label{eq:def_I_u}
I_{t_0}^{t_i} \left[(\vec u(t_j))_{j=1,\ldots,s}\right] =
\step \ds \sum_{j=1}^{s} a_{ij} \vec F\left(t_0+c_j \step, \vec u(t_j)\right);
\end{equation}
and hence, considering the quadrature formula for the IRK scheme, we obtain
\begin{equation}\label{eq:stage_order}
 \left\|\int_{t_{i-1}}^{t_i}\vec F(\vec u(\tau))\, \der \tau
 - I_{t_{i-1}}^{t_i}\left[(\vec u(t_j))_{j=1,\ldots,s}\right] \right\|
 \leq
 \eta_{i-1}^i := C_2 \step^{q+1},\qquad
i=1,\ldots,s,
\end{equation}
and for a stiffly accurate method,
\begin{equation*}
 \left\|\int_{t_{0}}^{t_s}\vec F(\vec u(\tau))\, \der \tau
 - I_{t_{0}}^{t_s}\left[(\vec u(t_j))_{j=1,\ldots,s}\right]\right\|
 \leq
 \eta_{0}^s := C_3 \step^{p+1},
\end{equation*}
recalling that $q$ and $p$ stand, respectively, for the stage order and
global order of the collocation method.
For the RadauIIA quadrature formulas, recall that $p=2s-1$ and $q=s$.

Denoting the exact solution to problem (\ref{eq:gen_prob_disc})
as $\vec u(t)=\Exsol^t \vec u_0$, we can in general write for the
splitting approximation that
\begin{equation*}
 \Exsol^{\step} \vec u_0 - \OS^{\step}\vec u_0 =
 c_{\widehat{p}+1}(\vec u_0)\step^{\widehat{p}+1}
 + c_{\widehat{p}+2}(\vec u_0)\step^{\widehat{p}+2} + \ldots,
\end{equation*}
recalling that $\widehat{p}$ stands for the order of approximation
of the splitting scheme
(see, e.g., \cite{MR2002152} Sect. IV.1 or \cite{MR2221614} Sect. III).
We then assume that the local truncation error of the
splitting approximation is bounded according to
\begin{equation}\label{eq:hyp_local}
 \left\|\Exsol^{\step} \vec u_0 - \OS^{\step}\vec u_0 \right\|
 \leq C_4 \step^{\widehat{p}+1},
\end{equation}
and that the following bound
\begin{equation}\label{eq:hyp_local_diff}
 \left\|\Exsol^{\step} \vec u_0 - \OS^{\step}\vec u_0 -
 \left[\Exsol^{\step} \vec v_0 - \OS^{\step} \vec v_0 \right] \right\|
 \leq C_5 \step^{\widehat{p}+1} \|\vec u_0-\vec v_0\|
\end{equation}
is also valid.
Moreover, we assume that $\OS^t$ satisfies the Lipschitz condition
\begin{equation}\label{eq:hyp_lip}
 \left\|\OS^{\step} \vec u_0 - \OS^{\step}\vec v_0 \right\|
 \leq (1+C_6\step) \|\vec u_0-\vec v_0\|.
\end{equation}

To simplify the discussion that follows, we define the error $e_i^k$
at each collocation node $i$ and iteration $k$ according to
\begin{equation}\label{eq:notation_e}
\left\|\vec u(t_i) - \widetilde{\vec u}^{k}_i \right\|
\leq e_i^k,\quad
i=1,\ldots,s,
\end{equation}
and the sum of these errors as
\begin{equation}
  \varXi^{k,s} := \ds \sum_{j=1}^s e_j^k.
\end{equation}
With these definitions, we prove the following theorem.
\begin{theorem}\label{teo:lemma1}
Considering problem (\ref{eq:gen_prob_disc}) with the Lipschitz condition for $\vec F(\vec u)$
(\ref{eq:lip_F}), the
DC--S iteration (\ref{eq2:def_utildeksplit})
with a given time step $\step$, and
assumptions
(\ref{eq:hyp_local_diff}) and (\ref{eq:hyp_lip})
for the splitting approximation, there are positive constants $A$, $B$,
such that for $k=1,2,\ldots$,
\begin{equation}\label{eq:lemma1}
\varXi^{k,s} \leq A + B \varXi^{k-1,s}.
\end{equation}
\end{theorem}
\begin{proof}
Defining $\step_i:=(c_i-c_{i-1})\step$,
$i=2,\ldots,s$,
we have from (\ref{eq2:def_utildeksplit}),
\begin{align*}
\vec u(t_i) - \widetilde{\vec u}^{k+1}_i =
& \vec u(t_i)
- \widetilde{\vec u}^{k}_{i-1} - I_{t_{i-1}}^{t_i} (\widetilde{\vec U}^k)
- \OS^{\step_i} \widetilde{\vec u}^{k+1}_{i-1}
+ \OS^{\step_i} \widetilde{\vec u}^{k}_{i-1} \nonumber \\
= & \int_{t_{i-1}}^{t_i}\vec F(\widetilde{\vec u}^{k}(\tau))\, \der \tau
- I_{t_{i-1}}^{t_i} (\widetilde{\vec U}^k)\nonumber \\
& + \Exsol^{\step_i} \vec u(t_{i-1}) - \OS^{\step_i} \vec u(t_{i-1}) -
 \left[\Exsol^{\step_i} \widetilde{\vec u}^{k}_{i-1} - \OS^{\step_i} \widetilde{\vec u}^{k}_{i-1} \right] \nonumber \\
 & + \OS^{\step_i} \vec u(t_{i-1}) - \OS^{\step_i} \widetilde{\vec u}^{k+1}_{i-1},
\end{align*}
with $k=1,2,\ldots$,
after adding and subtracting
$\Exsol^{\step_i} \widetilde{\vec u}^{k}_{i-1}
= \widetilde{\vec u}^{k}_{i-1} + \int_{t_{i-1}}^{t_i}\vec F(\widetilde{\vec u}^{k}(\tau))\, \der \tau$,
and $\OS^{\step_i} \vec u(t_{i-1})$;
and similarly,
\begin{align*}
\vec u(t_1) - \widetilde{\vec u}^{k+1}_1 =
& \vec u(t_1)
- \vec u_0 - I_{t_{0}}^{t_1} (\widetilde{\vec U}^k) \nonumber \\
= & \int_{t_{0}}^{t_1}\vec F(\vec u(\tau))\, \der \tau
 - I_{t_{0}}^{t_1}\left[(\vec u(t_j))_{j=1,\ldots,s}\right] \nonumber \\
 &+ I_{t_{0}}^{t_1}\left[(\vec u(t_j))_{j=1,\ldots,s}\right]
- I_{t_{0}}^{t_1} (\widetilde{\vec U}^k),
\end{align*}
after adding and subtracting $I_{t_{0}}^{t_1}\left[(\vec u(t_j))_{j=1,\ldots,s}\right]$.

Taking norms and considering (\ref{eq:stage_order}),
(\ref{eq:hyp_local_diff}), and (\ref{eq:hyp_lip}), there exist some
$\alpha$ and $\beta$ such that
\begin{equation}\label{eq:norm_i}
 \left\|\vec u(t_i) - \widetilde{\vec u}^{k+1}_i \right\| \leq
 \eta_{i-1}^i + \alpha \left\|\vec u(t_{i-1}) - \widetilde{\vec u}^{k}_{i-1} \right\|
 + \beta \left\|\vec u(t_{i-1}) - \widetilde{\vec u}^{k+1}_{i-1} \right\|,
\end{equation}
with $\alpha:=C_5 \step^{\widehat{p}+1}$ and
$\beta:=1+C_6\step$.
Similarly,
from (\ref{eq:diff_I}) and (\ref{eq:def_I_u})
there is a $\lambda:= C_1 \step$ such that
\begin{equation}\label{eq:norm_1}
 \left\|\vec u(t_1) - \widetilde{\vec u}^{k+1}_1 \right\| \leq
 \eta_{0}^1 + \lambda \
  \ds \sum_{j=1}^{s} \left\|\vec u(t_{j}) - \widetilde{\vec u}^{k}_{j} \right\|.
\end{equation}

Summing (\ref{eq:norm_i})--(\ref{eq:norm_1})
over $i$
and considering the notation
(\ref{eq:notation_e}), we have
for $i=2,\ldots,s$:
\begin{align}\label{eq:sum_e_s}
 \varXi^{k+1,i} & = \eta_{0}^i + \lambda \varXi^{k,s} + \alpha \varXi^{k,i-1} + \beta  \varXi^{k+1,i-1}
 \nonumber \\
 & \leq
 \eta_{0}^i + (\lambda + \alpha) \varXi^{k,s} + \beta  \varXi^{k+1,i-1}.
 \end{align}
In particular, from (\ref{eq:norm_1}) we have
\begin{equation}\label{eq:sum_e_1}
\varXi^{k,1} = e_1^k = \eta_{0}^1 + \lambda \varXi^{k-1,s}
\leq \eta_{0}^1 + (\lambda + \alpha) \varXi^{k-1,s},
\end{equation}
and using
the inequalities considered in
(\ref{eq:sum_e_s}) and (\ref{eq:sum_e_1})
with $\sumbeta^{i}:= \sum_{j=0}^i \beta^j$,
we obtain by mathematical induction over $i$,
\begin{equation}\label{eq:res_lemma1}
\varXi^{k,i} \leq
\sum_{j=1}^i \beta^{i-j} \eta_{0}^j
+  (\lambda + \alpha) \sumbeta^{i-1} \varXi^{k-1,s},
\end{equation}
which proves (\ref{eq:lemma1}).
\end{proof}

Bound (\ref{eq:lemma1}) in
Theorem \ref{teo:lemma1} accounts for the approximation errors
accumulated over the time subintervals
at a given iteration and for the way the sum of these errors
behaves from one iteration to the next one.
The next Corollary investigates the accumulation of errors after
a given number of iterations.
\begin{corollary}\label{teo:lemma2}
Considering
Theorem \ref{teo:lemma1},
assumption (\ref{eq:hyp_local})
for the splitting approximation
with $\gamma:= C_4 \step^{\widehat{p}+1}$,
and
$\sumbeta^{i}:= \sum_{j=0}^i \beta^j$ with
$\beta:=1+C_6\step$, we have for $k=1,2,\ldots$,
\begin{equation}\label{eq:lemma2}
\varXi^{k,s} \leq A \left[
1+ B + B^2 + \cdots + B^{k-1}\right] +
s\gamma \sumbeta^{s-1} B^k,
\end{equation}
with
$A= \sum_{j=1}^s \beta^{s-j} \eta_{0}^j$
and $B= (\lambda + \alpha) \sumbeta^{s-1}$.
\end{corollary}
\begin{proof}
Considering (\ref{eq:ini_split}) we have
\begin{align*}
\vec u(t_s) - \widetilde{\vec u}^{0}_s &=
 \vec u(t_s) - \OS^{\step_s} \widetilde{\vec u}^{0}_{s-1} \nonumber \\
&= \Exsol^{\step_s} \vec u(t_{s-1}) - \OS^{\step_s} \vec u(t_{s-1})
+ \OS^{\step_s} \vec u(t_{s-1}) - \OS^{\step_s} \widetilde{\vec u}^{0}_{s-1},
\end{align*}
and thus after taking norms,
\begin{equation*}
 e^0_s=\gamma + \beta e^0_{s-1} = \gamma \sumbeta^{s-1}.
\end{equation*}
Noticing that
\begin{equation*}
 \varXi^{0,s} = \gamma \left[ 1 + \sumbeta + \sumbeta^{2} + \cdots + \sumbeta^{s-1} \right]
 \leq s\gamma \sumbeta^{s-1},
\end{equation*}
bound (\ref{eq:lemma2}) follows by mathematical induction
over $k$ using (\ref{eq:lemma1}).
\end{proof}

With Theorem \ref{teo:lemma1} and
Corollary \ref{teo:lemma2}, the following can be proved.
\begin{theorem}\label{teo:teo1}
For the
DC--S scheme (\ref{eq:ini_split})--(\ref{eq2:def_utildeksplit})
with $k=1,2,\ldots$,
if $C_1<1$ in (\ref{eq:diff_I})
and $C_6 \step < 1$ in  (\ref{eq:hyp_lip}),
the following bound holds:
\begin{equation}\label{eq:res_teo1}
 \left\|\vec u(t_0+\step) - \widetilde{\vec u}^{k}_s \right\|
\leq
c \, \step^{\min[p+1,q+2,\widehat{p}+k+1]},
\end{equation}
with $c \geq \max \{C_3, C_2 C_6, s C_4 C_1 \}$.
\end{theorem}
\begin{proof}
 From (\ref{eq:res_lemma1}) in Theorem \ref{teo:lemma1} and considering that
 $\left\|\vec u(t_0+\step) - \widetilde{\vec u}^{k}_s \right\|
\leq e_s^k = \varXi^{k,s} - \varXi^{k,s-1}$, we obtain
\begin{align}\label{eq:teo_est1}
e^k_s \leq &
\sum_{j=1}^s \beta^{s-j} \eta_{j-1}^j
+  (\lambda + \alpha) \beta^{s-1} \varXi^{k-1,s}
\nonumber \\
 \leq &
\sum_{j=1}^s \beta^{s-j} \eta_{j-1}^j
+
(\lambda + \alpha) \beta^{s-1} \left(
A \left[
1+ B + B^2 + \cdots + B^{k-2}\right] \right)
\nonumber \\
&
+
(\lambda + \alpha) \beta^{s-1}
s\gamma \sumbeta^{s-1} B^{k-1},
\end{align}
according to Corollary \ref{teo:lemma2}
with
$A= \sum_{j=1}^s \beta^{s-j} \eta_{0}^j$
and $B= (\lambda + \alpha) \sumbeta^{s-1}$.
Recalling that $\beta=1+C_6\step$,
$\eta_{j-1}^j = C_2 \step^{q+1}$,
$j=1,\ldots,s$,
and
$\eta_{0}^s = C_3 \step^{p+1}$,
and considering the binomial series,
we have that
\begin{align}\label{eq:DCS_quad_order}
 \sum_{j=1}^s \beta^{s-j} \eta_{j-1}^j &=
 \sum_{j=1}^s \left[ \sum_{l=0}^{\infty}
 \left(\begin{array}{c}
s-j\\
l
\end{array}\right)
 (C_6\step)^l
 \right]
 \eta_{j-1}^j
 \nonumber \\
 &=
 \eta_{0}^s + \sum_{j=1}^{s-1}
 \left[
 \sum_{l=1}^{\infty}
 \left(\begin{array}{c}
s-j\\
l
\end{array}\right)
 (C_6\step)^l
 \right]
 \eta_{j-1}^j
 \nonumber \\
 &=
C_3 \step^{p+1} +
 C_2 \step^{q+1}
 \sum_{j=1}^{s-1}
 \sum_{l=1}^{\infty}
 \left(\begin{array}{c}
s-j\\
l
\end{array}\right)
 (C_6\step)^l;
 \end{align}
 and similarly,
 \begin{equation*}
  A = \sum_{j=1}^s \beta^{s-j} \eta_{0}^j =
  C_3 \step^{p+1} +
 C_2 \step^{q+1} (1+C_6\step)
 \sum_{l=1}^{\infty}
 \left(\begin{array}{c}
s-1\\
l
\end{array}\right)
 (C_6\step)^{l-1},
 \end{equation*}
 where we have considered that
 \begin{equation*}
 \sumbeta^{s-1} = \sum_{j=0}^{s-1} \beta^j =
 \frac{1-\beta^s}{1-\beta} =
 \sum_{l=1}^{\infty}
 \left(\begin{array}{c}
s\\
l
\end{array}\right)
 (C_6\step)^{l-1}.
 \end{equation*}

 Since
 $\lambda = C_1 \step$ and
 $\alpha=C_5 \step^{\widehat{p}+1}$,
 we know that
 \begin{equation*}
 (\lambda + \alpha) \beta^{s-1} = (C_1 \step + C_5 \step^{\widehat{p}+1} )
 \sum_{l=0}^{\infty}
 \left(\begin{array}{c}
s-1\\
l
\end{array}\right)
 (C_6\step)^{l}
 \end{equation*}
 and
 \begin{equation*}
 B= (\lambda + \alpha) \sumbeta^{s-1} =
 (C_1 \step + C_5 \step^{\widehat{p}+1} )
 \sum_{l=1}^{\infty}
 \left(\begin{array}{c}
s\\
l
\end{array}\right)
 (C_6\step)^{l-1};
 \end{equation*}
we thus have
with $\gamma= C_4 \step^{\widehat{p}+1}$
that
\begin{equation}\label{eq:DCS_split_order}
(\lambda + \alpha) \beta^{s-1} s\gamma \sumbeta^{s-1} B^{k-1} =
 s C_4 C_1^k \step^{\widehat{p}+k+1}
 + \Or(C_1^k \step^{\widehat{p}+k+2}),
\end{equation}
 which together with (\ref{eq:DCS_quad_order})
 and (\ref{eq:teo_est1})
 prove
 (\ref{eq:res_teo1})
 with $C_1<1$ into (\ref{eq:DCS_split_order}),
taking into account that the expression
$(\lambda + \alpha) \beta^{s-1}
A \left[
1+ B + B^2 + \cdots + B^{k-2}\right]$
yields
$\Or \left( \step^{\min[p+2,q+2]} \right)$
plus higher order terms.
\end{proof}

For RadauIIA quadrature formulas, bound (\ref{eq:res_teo1}) reads
$c \, \step^{\min[2s,s+2,\widehat{p}+k+1]}$.
The impact of approximating the solution at the collocation nodes of
the quadrature formula can be seen, for instance,
in (\ref{eq:DCS_quad_order}), where a lower order
might be attained
as a consequence of
the intermediate stage approximations.
In particular the DC--S scheme (\ref{eq2:def_utildeksplit})
does not necessarily converge to the IRK scheme.
On the other hand,
(\ref{eq:DCS_split_order}) accounts for
the iterative corrections performed on the initial
low order approximation.
Notice that in practice a limited number of iterations $k$ will be
performed, certainly less than $p$, and therefore the
scheme would behave like $\Or \left( \step^{\min[2s,s+2,\widehat{p}+k+1]} \right)$
even for a finite $C_1\ge 1$ in Theorem~\ref{teo:teo1}.
The following Corollary gives us some further insight
into the behavior of the DC--S scheme.
\begin{corollary}\label{teo:teo2}
 There is a maximum time step
 $\step_{\max}$ such that for a given
 $\step < \step_{\max}$, there is a positive
 and bounded $C(k)$ such that for  $k=1,2,\ldots, \min [p-\widehat{p},q-\widehat{p}+1]$,
 the following holds:
\begin{equation}\label{eq:res_teo2}
 \left\|\vec u(t_0+\step) - \widetilde{\vec u}^{k}_s \right\|
\leq
C(k) \left\|
\widetilde{\vec u}^k_s - \widetilde{\vec u}^0_s
\right\|.
\end{equation}
\end{corollary}
\begin{proof}
For successive iterations $k$ and $k-1$ such that $k \leq \min [p-\widehat{p},q-\widehat{p}+1]$, we note that (\ref{eq:teo_est1}) reduces to (\ref{eq:DCS_split_order})
and hence there
is a positive constant $\zeta_k$ such that
\begin{equation}\label{eq:def_alphak}
\left\| \vec u(t_0+\step) - \widetilde{\vec u}^{k}_s \right\| =
\zeta_k \step
\left\| \vec u(t_0+\step) - \widetilde{\vec u}^{k-1}_s \right\|,
\end{equation}
and hence
\begin{equation}\label{eq:def_gammak}
\left\| \vec u(t_0+\step) - \widetilde{\vec u}^{k}_s \right\| =
\sigma_k (\step)^k
\left\| \vec u(t_0+\step) - \widetilde{\vec u}^0_s \right\|
\end{equation}
where
$\sigma_k = \Pi _{j=1}^k \zeta_j$.
Combining (\ref{eq:def_gammak})
with
\begin{equation*}\label{eq:err0errk}
\left\| \vec u(t_0+\step) - \widetilde{\vec u}^0_s \right\| \leq
\left\| \vec u(t_0+\step) - \widetilde{\vec u}^k_s \right\| +
\left\| \widetilde{\vec u}^k_s - \widetilde{\vec u}^0_s \right\|,
\end{equation*}
yields (\ref{eq:res_teo2}) with
\begin{equation}\label{eq:errk_bound}
C(k) =
\frac{\sigma_k (\step)^k}{1-\sigma_k (\step)^k}
\end{equation}
and
$\step_{\max}=(\max_k \zeta_k)^{-1}$
such that $C(k)>0$ for
any given $\step < \step_{\max}$.
\end{proof}

Notice that (\ref{eq:def_alphak})
explicitly reflects the increase of the approximation
order with every correction iteration established in
Theorem \ref{teo:teo1}.
A maximum time step per iteration
$\step_{\max,k}$ can be also defined
as $\step_{\max,k}=(\zeta_k)^{-1}$,
which in particular implies that
no correction is expected
for $\step \geq \step_{\max,k}$
according to (\ref{eq:def_alphak}).
The maximum time step is thus given by
$\step_{\max}=\min_k \step_{\max,k}$
in Corollary \ref{teo:teo2}.  In practice we note that
$\step_{\max,k}$ can be larger than
$\step_{\max}$ during a given iteration.

\section{Time stepping and error control}\label{sec:error_control}
Since the numerical integration within the DC--S method
is in practice performed by a splitting solver
with no stability constraints, we introduce a time step
selection strategy based on a user--defined accuracy tolerance.
Denoting
$err_k := \| \vec u(t_0+\step) - \widetilde{\vec u}^k_s \|$,
as the approximation error of the DC--S scheme (\ref{eq2:def_utildeksplit})
after a time step $\step$,
Corollary \ref{teo:teo2} gives us an estimate of
$err_k$, where in particular
$err_0$ stands for the initial splitting error.
In practice we can numerically estimate the error by computing
(\ref{eq:res_teo2}) with (\ref{eq:errk_bound}),
for which we first need to estimate the $\sigma_k$ coefficients.

Let us introduce the following approximation to
$\vec u(t_0+\step)$
according to (\ref{eq2:runge_kutta_recast}),
based on the original $s$--stage Runge--Kutta scheme:
\begin{equation}\label{eq2:def_ubar}
\bar{\vec u}^k_s = \vec u_0 + I_{t_0}^{t_s} (\widetilde{\vec U}^k).
\end{equation}
Notice that $\bar{\vec u}^k_s$ is in general different from
$\widehat{\vec u}^k_s$
computed according to (\ref{eq2:def_uhat}).
In particular, for sufficiently small
$\step$, $\bar{\vec u}^0_s$ should be a better approximation to
$\vec u(t_0+\step)$ than the splitting solution $\widetilde{\vec u}^0_s$ given by (\ref{eq:ini_split}).
Therefore, we assume that
\begin{equation}\label{eq:split_err}
\left\| \vec u(t_0+\step) - \widetilde{\vec u}^0_s \right\| \leq
\left\|  \vec u(t_0+\step) -  \bar{\vec u}^0_s \right\| +
\left\| \bar{\vec u}^0_s - \widetilde{\vec u}^0_s \right\| \approx
\left\| \bar{\vec u}^0_s - \widetilde{\vec u}^0_s \right\|,
\end{equation}
and suppose that the same property holds for the
corrective iterations,
taking also into account that
the fully coupled system
$\vec F(\widetilde{\vec u}^k(t))$
is evaluated in (\ref{eq2:def_ubar})
for $\bar{\vec u}^k_s$.
Introducing
$\overline{err}_k = \| \bar{\vec u}^k_s - \widetilde{\vec u}^k_s \|$,
we approximate $\zeta_k$ in (\ref{eq:def_alphak}) by
\begin{equation*}
\widetilde{\zeta}_k =
\frac{1}{\step}
\frac{\overline{err}_k}{\overline{err}_{k-1}}, \qquad
k=1,2,\ldots,
\end{equation*}
and $\sigma_k$ by $\widetilde{\sigma}_k = \Pi _{j=1}^k \widetilde{\zeta}_j$.
We then estimate the error
of the DC--S method,
$\widetilde{err}_k$,
according to (\ref{eq:res_teo2}):
\begin{equation}\label{eq:errk_bound2}
\widetilde{err}_k =
\left[
\frac{\widetilde{\sigma}_k (\step)^k}{1-\widetilde{\sigma}_k (\step)^k}
\right]
\left\| \widetilde{\vec u}^k_s - \widetilde{\vec u}^0_s \right\|, \qquad
k=1,2,\ldots.
\end{equation}
The initial error, that is, the splitting error
$err_0$ can be directly approximated by $\overline{err}_0$,
following (\ref{eq:split_err}),
and thus $\widetilde{err}_0=\overline{err}_0$.

Having an estimate of the approximation error
(\ref{eq:errk_bound2}), we can define
an accuracy tolerance $\eta$ such that
for a user--defined $\eta$,
no further correction iterations are performed
if $\widetilde{err}_k \leq \eta$
is satisfied.
Therefore,
by supposing that
\begin{equation*}\label{eq:eta_dtnew}
\eta =
\left[
\frac{\widetilde{\sigma}_{k} (\step_{\new,k})^k}
{1-\widetilde{\sigma}_{k} (\step_{\new,k})^k}
\right]
\left\|\widetilde{\vec u}^k_s - \widetilde{\vec u}^0_s \right\|
\end{equation*}
and comparing it to (\ref{eq:errk_bound2}),
we have that
\begin{equation}\label{eq:dtnew}
\step_{\new,k} =
\left[
\frac{\eta}
{\left[1-\widetilde{\sigma}_{k} (\step)^k\right]\widetilde{err}_k
+ \widetilde{\sigma}_{k} (\step)^k\eta}
\right]^{1/k}
\step.
\end{equation}
We can thus estimate the
new time step as $\step_{\new}=\nu
\times
\min (\step_{\new,k},\step_{\max,k})$,
after $k$ iterations,
where
$\step_{\max,k}=(\widetilde{\zeta}_k)^{-1}$
and $\nu$ ($0<\nu\leq 1$) is a security factor.
Note that this new time step supposes that $k$
correction iterations will also be used for the next time step.
By defining a maximum number of iterations
$k_{\max}$, another way of estimating the
new time step supposes that
\begin{equation}\label{eq:eta_dtnew_kmax}
\eta =
\left[
\frac{\widetilde{\sigma}_{k_{\max}} (\step_{\new,k_{\max}})^{k_{\max}}}
{1-\widetilde{\sigma}_{k_{\max}} (\step_{\new,k_{\max}})^{k_{\max}}}
\right]
\left\|\widetilde{\vec u}^k_s - \widetilde{\vec u}^0_s \right\|,
\end{equation}
and hence,
\begin{equation}\label{eq:dtnew_kmax}
\step_{\new,k_{\max}} =
\left[
\frac{\eta}
{\left[1-\widetilde{\sigma}_{k} (\step)^k\right]\widetilde{err}_k
+ \widetilde{\sigma}_{k} (\step)^k\eta}
\right]^{1/k_{\max}}
\left[
\frac{\widetilde{\sigma}_{k}}{\widetilde{\sigma}_{k_{\max}}}
\right]^{1/k_{\max}}
\step^{k/k_{\max}},
\end{equation}
where $\widetilde{\sigma}_{k_{\max}}$ is approximated
by $(\widetilde{\zeta}_k)^{k_{\max}-k}\widetilde{\sigma}_{k}$.
In this case the new time step can be computed
as $\step_{\new}=\nu \times
\min (\step_{\new,k_{\max}},\step_{\max,k})$,
with the maximum number of iterations
given, for instance, by
$k_{\max}=\min [p-\widehat{p},q-\widehat{p}+1]$.
Notice that the time step $\step_{\new,k_{\max}}$
is potentially larger than $\step_{\new,k}$
for $k<k_{\max}$, since the former needs more iterations to
attain an accuracy of $\eta$ according to (\ref{eq:eta_dtnew_kmax}).
The best choice between $\step_{\new,k}$ and $\step_{\new,k_{\max}}$
depends largely on the problem and, in particular, on the corresponding
computational costs associated, for instance, with the
function evaluations or with the splitting procedure.
In general both approaches look for approximations
of accuracy $\eta$, where smaller time steps involve
fewer correction iterations and {\it vice--versa}.
We could also consider, for instance,
$\step_{\new}=\nu \times
\min (\step_{\new,k},\step_{\new,k_{\max}},\step_{\max,k})$.

Recalling that the splitting error can be estimated as
$\widetilde{err}_0=\overline{err}_0=\| \bar{\vec u}^0_s - \widetilde{\vec u}^0_s \|$,
we can also define a new way to adapt the splitting time step
for stiff problems, as previously considered in
\cite{MR2847238}.
Given the splitting approximation (\ref{eq:ini_split}),
one needs only
to compute $\bar{\vec u}^0_s$, according to (\ref{eq2:def_ubar}).
Defining $\eta_{\splt}$ as the accuracy tolerance for the
splitting approximation, we can compute the splitting time step
$\step_{\new,\splt}$ as
\begin{equation}\label{eq:dtnew_split}
\step_{\new,\splt} =
\nu
\left[
\frac{\eta_{\splt}}
{\widetilde{err}_0}
\right]^{1/(\widehat{p}+1)}
\step,
\end{equation}
assuming from (\ref{eq:hyp_local}) that
$\widetilde{err}_0= C_4 \step^{\widehat{p}+1}$
and
$\eta = C_4 \step_{\new,\splt}^{\widehat{p}+1}$,
where $\widehat{p}$ stands for the order of the splitting scheme.
The same technique may be applied to
any other low order time integration method.
In particular by setting $\eta_{\splt}=\eta$,
we can estimate
the splitting time step $\step_{\new,\splt}$
that would be required to attain
accuracy $\eta$ through (\ref{eq:dtnew_split}), in contrast to
$\step_{\new,k}$ or $\step_{\new,k_{\max}}$.

\subsection{Predicting approximation errors}
If $k_{\max}$ correction iterations have been
performed and the error estimate is still too large ($\widetilde{err}_{k_{\max}} > \eta$), then
the DC--S scheme defined in (\ref{eq2:def_utildeksplit}),
together with the initial splitting approximation
(\ref{eq:ini_split}), must be restarted with a new time
step $\step_{\new}$, computed as previously defined.  In this case, these $k_{\max}$ iterations are wasteful, and hence we derive a method for predicting $\widetilde{err}_{k_{\max}}$ based on the current error estimate $\widetilde{err}_k$.
Denoting the predicted final error estimate
as $\widetilde{err}^{*}_{k_{\max}}$,
we approximate it according to
\begin{equation*}\label{eq:pred_errk}
\widetilde{err}^{*}_{k_{\max}} =
(\widetilde{\zeta}_k \step)^{k_{\max}-k}
\widetilde{err}_k, \qquad
k=1,\ldots,k_{\max}-1,
\end{equation*}
following (\ref{eq:def_alphak}).
If for some $k< k_{\max}-1$,
we have that
$\widetilde{err}^{*}_{k_{\max}}> \eta$,
we restart the scheme
(\ref{eq:ini_split})--(\ref{eq2:def_utildeksplit})
with a new time step given by
\begin{equation*}\label{eq:pred_dtnew_kmax}
\step^{*}_{\new,k_{\max}} =
\nu
\left[
\frac{\eta}
{(\widetilde{\zeta}_k)^{k_{\max}-k}
\widetilde{err}_k}
\right]^{1/k_{\max}}
\step^{k/k_{\max}},
\end{equation*}
supposing that
$\widetilde{err}_{0,\new} = \widetilde{err}_0$ in
\begin{equation*}
\eta =
\widetilde{\sigma}_{k_{\max}}(\step^{*}_{\new,k_{\max}})^{k_{\max}}
\left[
\frac{\widetilde{err}_k}
{\widetilde{\sigma}_{k} (\step)^k}
\right],
\end{equation*}
where
$\widetilde{\sigma}_{k_{\max}}=
(\widetilde{\zeta}_k)^{k_{\max}-k}\widetilde{\sigma}_{k}$.
For the initial approximation $k=0$,
we predict $\widetilde{\sigma}_{k_{\max}}$ based on the
previous $\widetilde{\sigma}_{k_{\max},\old}$
and time step $\step_{\old}$ according to
\begin{equation*}\label{eq:pred_err0}
\widetilde{err}^{*}_{k_{\max}} =
\widetilde{\sigma}^{*}_{k_{\max}}
(\step)^{k_{\max}}
\widetilde{err}_0, \qquad
\widetilde{\sigma}^{*}_{k_{\max}}=
\widetilde{\sigma}_{k_{\max},\old}
\left[
\frac{\step}
{\step_{\old}}
\right]^{k_{\max}};
\end{equation*}
with the corresponding time step
\begin{equation*}\label{eq:pred_dtnew_kmax0}
\step^{*}_{\new,k_{\max}} =
\nu
\left[
\frac{\eta}
{\widetilde{\sigma}^{*}_{k_{\max}}
\widetilde{err}_0}
\right]^{1/k_{\max}},
\end{equation*}
supposing again that $\widetilde{err}_{0,\new} = \widetilde{err}_0$.

\section{Space discretization errors}\label{sec:space}
So far, only temporal discretization errors were investigated.
For the sake of completeness
we briefly study in the following the impact of spatial discretization
errors in the approximations computed with the DC--S method.
Contrary to Section~\ref{sec:analysis}, only an empirical approach
is considered.
Denoting by
$u(t_0+\step,\x)\vert_{\X}$, the analytic
solution of problem (\ref{eq:gen_prob})
projected on the grid $\X$ at time $t_0+\step$, and by
$\vec u(t_s)$, the exact solution of the semi--discrete problem
(\ref{eq:gen_prob_disc}) at time $t_s=t_0+\step$;
we can assume that for the approximation $\widetilde{\vec u}^k_s$
coming from (\ref{eq2:def_utildeksplit}),
the following bound is satisfied
\begin{equation}\label{eq:err_uk_disc}
\left\| u(t_0+\step,\x)\vert_{\X}- \widetilde{\vec u}^k_s \right\| \leq
\left\| u(t_0+\step,\x)\vert_{\X} - \vec u(t_s) \right\| +
\left\| \vec u(t_s) - \widetilde{\vec u}^k_s \right\|.
\end{equation}
That is, the approximation error of the DC--S
method
is bounded by the sum of space discretization errors,
$\| u(t_0+\step,\x)\vert_{\X} - \vec u(t_s) \|$,
and time discretization errors,
$\| \vec u(t_s) - \widetilde{\vec u}^k_s \|$.
As previously established in
Theorem \ref{teo:teo1},
the local order of approximation of the latter is given by
$\min[p+1,q+2,\widehat{p}+k+1]$.
Therefore,
given a semi--discrete problem
(\ref{eq:gen_prob_disc}) with solution $\vec u(t)$,
high order approximations can be computed
by means of (\ref{eq2:def_utildeksplit})
within an accuracy tolerance of $\eta$.
In particular, for sufficiently fine grids
and/or if high order space discretization schemes
are used,
space discretization errors may become
small enough and the approximation
error of the DC--S method would be estimated as $\eta$, also
with respect to the analytic solution of problem (\ref{eq:gen_prob})
according to (\ref{eq:err_uk_disc}).

\subsection{Introducing high order space discretization}\label{subsubsec:high_order_space}
Spatial resolution is a critical
aspect for many problems, and high order space discretization schemes are
often required.
We consider high order schemes to discretize
$F(u)$ in space
for problem (\ref{eq:gen_prob}),
and denote such an approximation as $\vec F^{\HO}(\vec u)$,
so that
$\vec u^{\HO}(t)$ stands for the solution of the
corresponding semi--discrete problem.
We can thus use scheme (\ref{eq2:def_utildeksplit})
to approximate $\vec u^{\HO}(t)$ by considering
$\vec F^{\HO}(\vec u)$ in the computation of the
$\widehat{\vec u}^k_i$'s in (\ref{eq2:def_uhat}),
and $\vec F^{\HO}_i(\vec u)$, $i=1,2,\ldots$,
for the splitting approximations in (\ref{eq2:def_utildeksplit})
and (\ref{eq:ini_split}).
Let us call this approximation $(\widetilde{\vec u}^k_s)^{\HO}$.
Recalling that the
numerical  error
$\| \vec u^{\HO}(t_s) - (\widetilde{\vec u}^k_s)^{\HO} \|$
can be controlled by an accuracy tolerance $\eta$,
a high order space discretization also results in
better error control with respect to the analytic solution
$u(t,\x)$, following the discussion in the previous section.

A more efficient alternative, however, uses
$\vec F^{\HO}(\vec u)$ only in the quadrature formulas
for the $\widehat{\vec u}^k_i$'s
in (\ref{eq2:def_uhat}), and
low order discrete $\vec F_i(\vec u)$'s, $i=1,2,\ldots$
for the splitting approximations in (\ref{eq2:def_utildeksplit})
and (\ref{eq:ini_split}).
The latter procedure results in an approximation
$(\widetilde{\vec u}^k_s)^{\widetilde{\HO}}$
to $(\widetilde{\vec u}^k_s)^{\HO}$.
Since low order discretization in space is used to evaluate
the corrections,
that is, to predict the solutions at the collocation nodes of the quadrature formulas,
we expect that
$(\widetilde{\vec u}^k_s)^{\widetilde{\HO}}$
becomes equivalent to
$(\widetilde{\vec u}^k_s)^{\HO}$
after some iterations;
if this does happen, then $(\widetilde{\vec u}^k_s)^{\widetilde{\HO}}$
becomes an approximation of maximum order $p=2s-1$ in time,
according  to Theorem \ref{teo:teo1},
and of high order in space
to the analytic solution
$u(t,\x)$ of problem (\ref{eq:gen_prob}).
Nevertheless, taking into account the hybrid structure
of $(\widetilde{\vec u}^k_s)^{\widetilde{\HO}}$ in terms of space
discretization errors, a simple decomposition of errors
like (\ref{eq:err_uk_disc}) is no longer possible.
Moreover, we cannot expect that each correction iteration
increases the order of approximation by one and consequently,
the error estimates previously established
in \S\ref{sec:error_control}
are no longer
valid for this configuration.

\section{Numerical illustrations: The Belousov--Zhabotinski reaction}\label{sec:num_res}
Let us consider the numerical approximation of a model for
the Belousov--Zha\-bo\-tins\-ki (BZ) reaction, a catalyzed oxidation of an organic species by acid
bromated ion
(see \cite{Epstein98} for more details and illustrations).
The present mathematical formulation \cite{Field72,Scott94}
takes into account three species: hypobromous acid $\mathrm{HBrO_2}$,
bromide ions $\mathrm{Br^-}$, and cerium (IV).
Denoting by $a=[\mathrm{Ce(IV)}]$, $b=[\mathrm{HBrO_2}]$, and $c=[\mathrm{Br^-}]$,
we obtain a very stiff system of three PDEs given by
\begin{equation} \label{eq4:bz_eq_3var_diff}
\left.
\begin{array}{l}
\partial_t a - D_a\, \partial^2 _\x a = \ds \frac{1}{\mu}(-qa-ab+fc),\\[1.75ex]
\partial_t b - D_b\, \partial^2 _\x b = \ds \frac{1}{\varepsilon}\left(qa-ab+b(1-b)\right),\\[1.75ex]
\partial_t c - D_c\, \partial^2 _\x c = b-c,
\end{array}
\right\}
\end{equation}
where $\x \in \R^d$,
with real, positive
parameters: $f$, small $q$, and small $\varepsilon$ and
$\mu$, such that  $\mu \ll \varepsilon \ll 1$.
In this study: $\varepsilon = 10^{-2}$,
$\mu = 10^{-5}$, $f=1.6$, $q=2\times 10^{-3}$;
with diffusion coefficients:
$D_a=2.5\times 10^{-3}$, $D_b=2.5\times 10^{-3}$, and $D_c=1.5\times 10^{-3}$.
The dynamical system associated with this problem
models reactive, excitable media with a
large time scale spectrum (see \cite{Scott94}
for more details).
The spatial configuration with
the addition of diffusion develops propagating wavefronts
with steep spatial gradients.

\subsection{Numerical analysis of errors}\label{subsec:temp_err}
We consider problem
(\ref{eq4:bz_eq_3var_diff}) in a 1D configuration,
discretized on a uniform grid of 1001 points
over a space region of $[0,1]$.
A standard, second order, centered finite differences scheme
is employed for the diffusion term.
To obtain an initial condition, we initialize the problem
with a discontinuous profile close to the left boundary;
we then integrate until the BZ wavefronts are fully developed.
Figure~\ref{fig:sol_BZ_nx1001} shows the time evolution of the
propagating waves for a time window of $[0,1]$.
For comparison,
we can solve the
semi--discrete problem associated with (\ref{eq4:bz_eq_3var_diff})
using a dedicated solver for stiff ODEs.
We consider the Radau5 solver \cite{Hairer96},
based on a fifth order,
implicit Runge--Kutta scheme built upon
the RadauIIA coefficients in (\ref{eq:radau5}).
The reference solution for the
semi--discrete problem is thus given by
the solution obtained with Radau5 over the time interval $[0,1]$,
computed with a fine tolerance: $\eta_{\rm Radau5} = 10^{-14}$.
\begin{figure}[!htb]
\begin{center}
\includegraphics[width=0.49\textwidth]{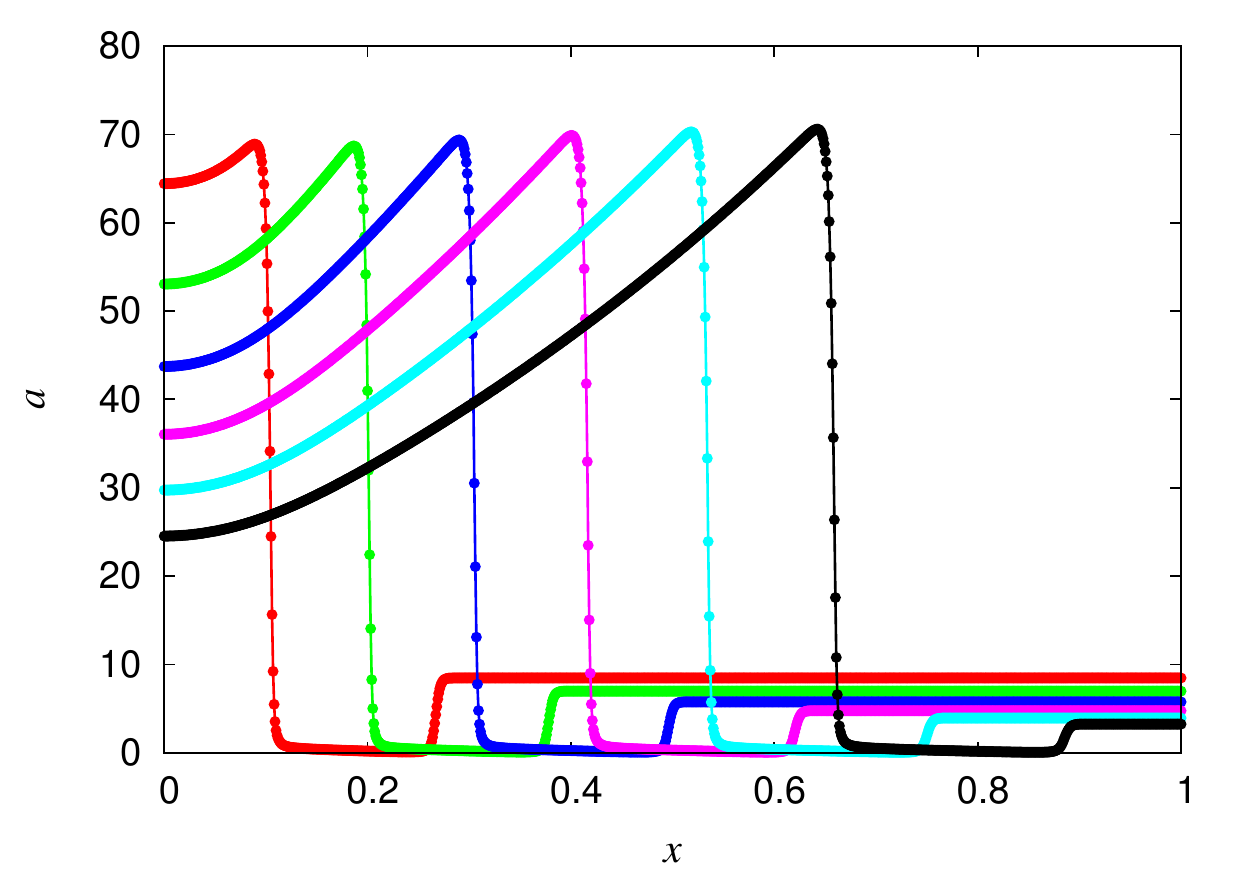}
\includegraphics[width=0.49\textwidth]{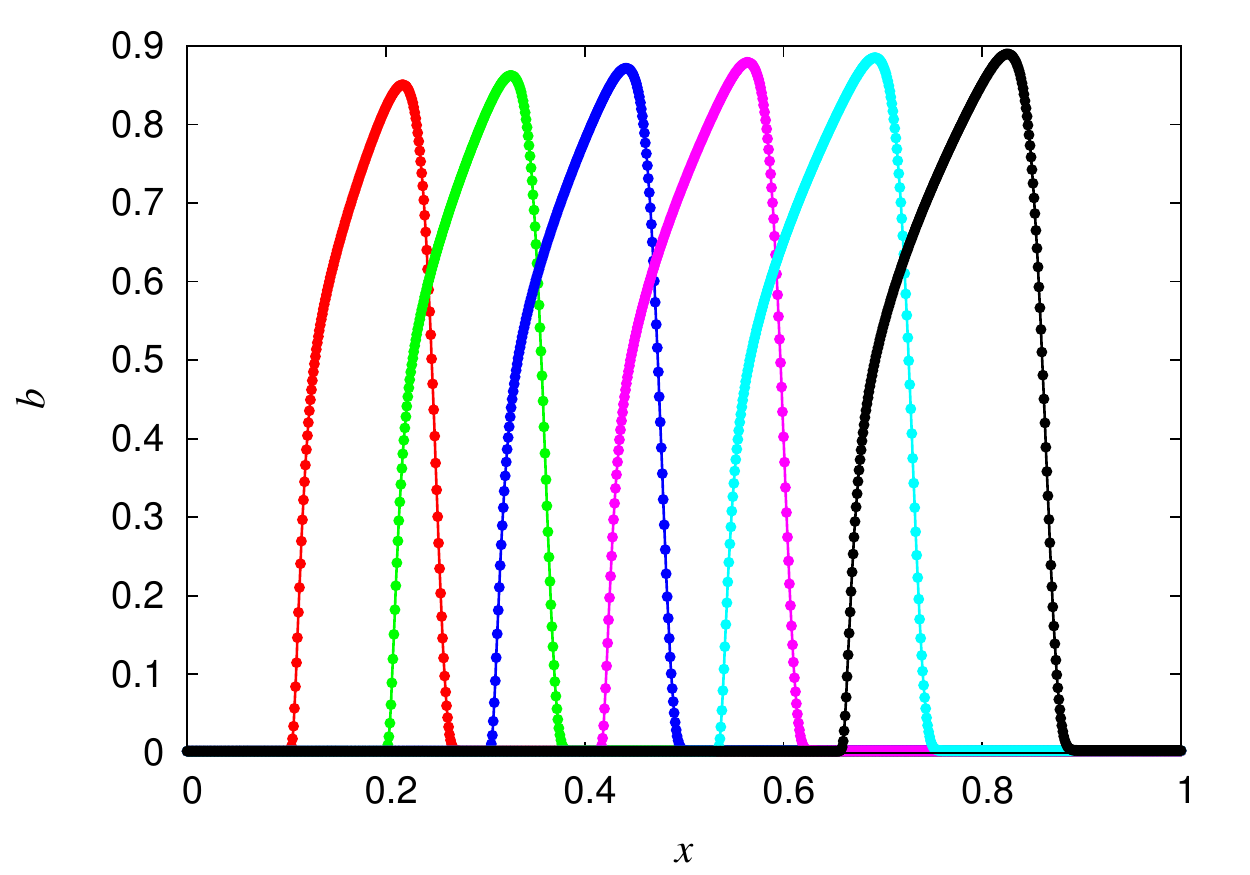}
\includegraphics[width=0.49\textwidth]{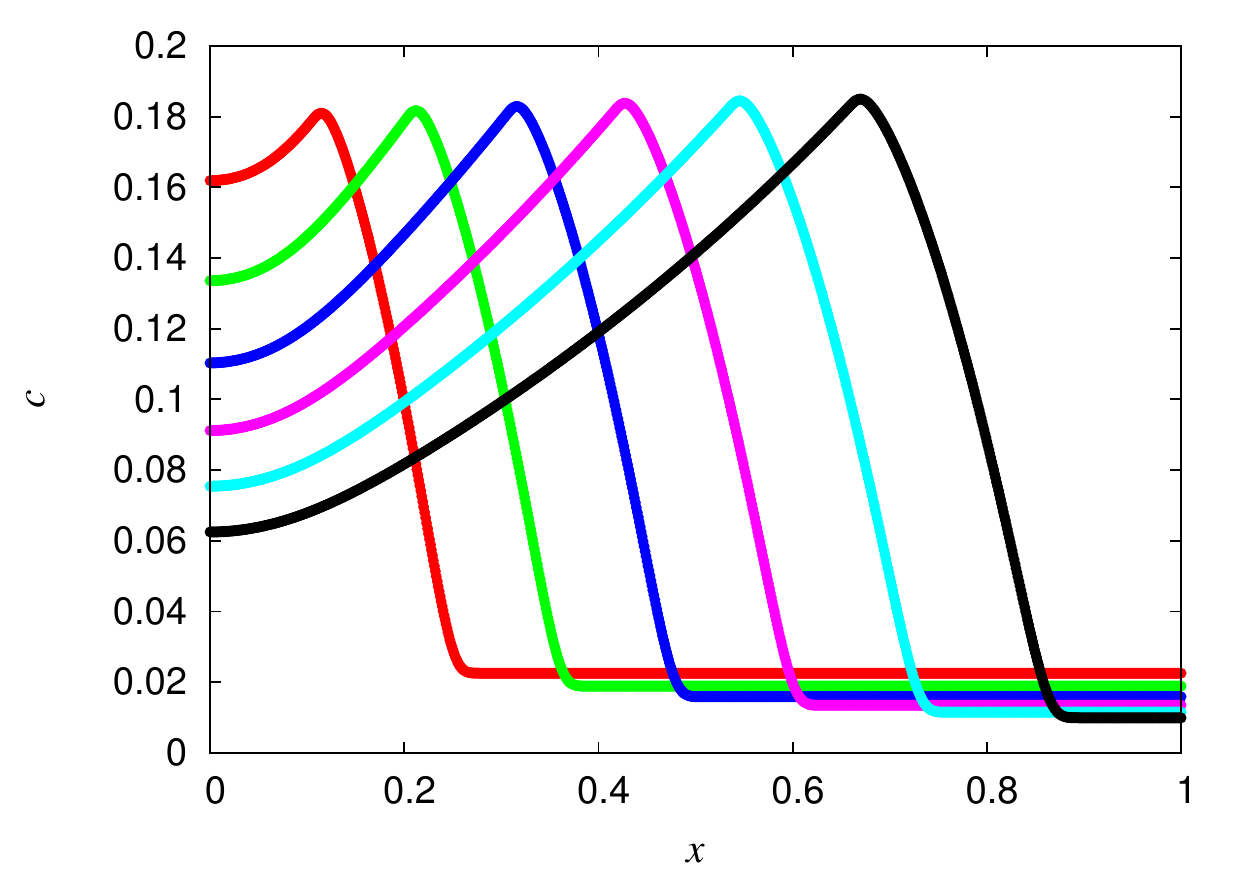}
\end{center}
\caption{BZ propagating waves for variables $a$ (top left),
$b$ (top right), and $c$ (bottom), at time intervals of $0.2$
within $[0,1]$ from left to right.}
\label{fig:sol_BZ_nx1001}
\end{figure}

We consider the splitting solver introduced in
\cite{Duarte11_SISC}
for the initialization of the iterative algorithm
(\ref{eq:ini_split}) and the corrective stages
in (\ref{eq2:def_utildeksplit})
for the DC--S scheme.
Radau5 is used to integrate point by point
the chemical terms in (\ref{eq4:bz_eq_3var_diff}),
while the diffusion problem is integrated with a
fourth order, stabilized explicit Runge--Kutta method:
ROCK4 \cite{MR1923724}.
(See \cite{Duarte11_ESAIM,Dumont2013} and discussions therein
for a more complex application considering the same splitting
solver.)
Tolerances of these solvers are set to a
reasonable value of
$\eta_{\rm Radau5} = \eta_{\rm ROCK4} = 10^{-5}$,
taking into account that the splitting solver
only provides approximations to the solutions at the collocation
nodes used in the high order quadrature formulas.
Again, the RadauIIA coefficients of order 5
given in (\ref{eq:radau5}), that is, with
$s=3$ nodes,
are considered for the
quadrature formulas in (\ref{eq2:def_uhat}).
In what follows
we consider both Lie and Strang splitting schemes,
(\ref{eq:lie}) and (\ref{eq:strang}),
ending with the time integration
of the reaction terms.
The latter is particularly relevant for relatively
large splitting time steps \cite{Descombes04,Descombes_LM}.
\begin{figure}[!htb]
\begin{center}
\includegraphics[width=0.49\textwidth]{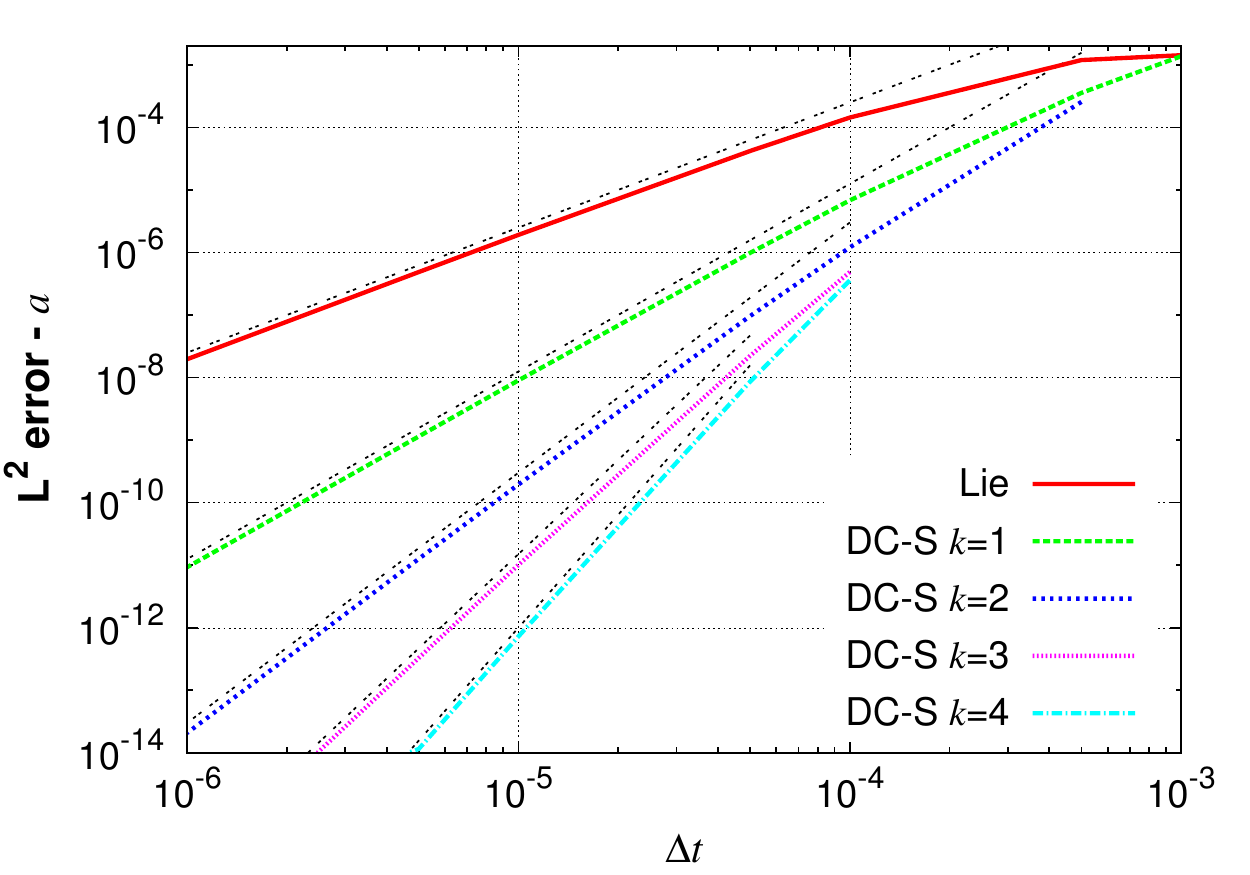}
\includegraphics[width=0.49\textwidth]{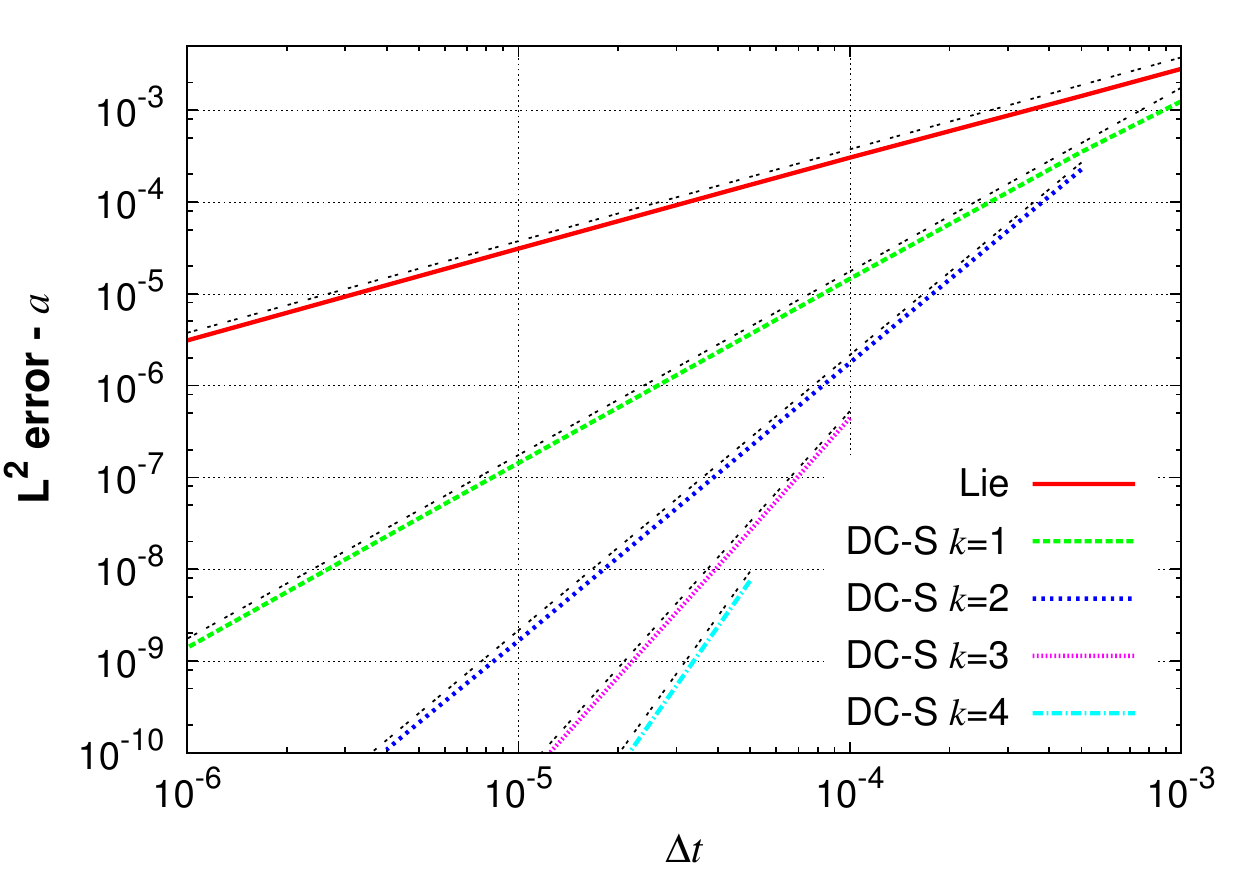}
\end{center}
\caption{Local (left) and global (right) $L^2$--errors
for the DC--S scheme with Lie splitting.
Dashed lines of slopes 2 to 6 (left), and 1 to 5 (right)
are also depicted.}
\label{fig:err_lie_nx1001}
\end{figure}
\begin{figure}[!htb]
\begin{center}
\includegraphics[width=0.49\textwidth]{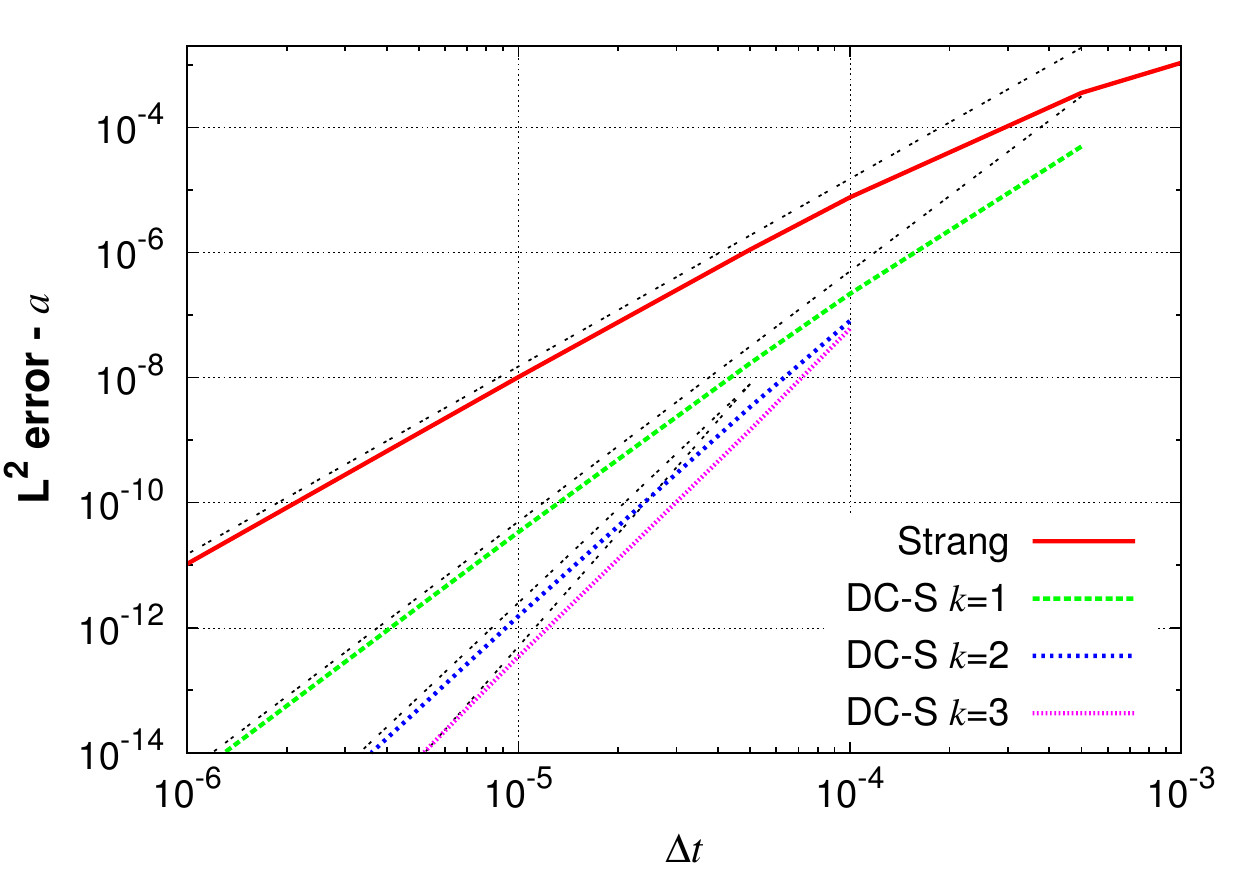}
\includegraphics[width=0.49\textwidth]{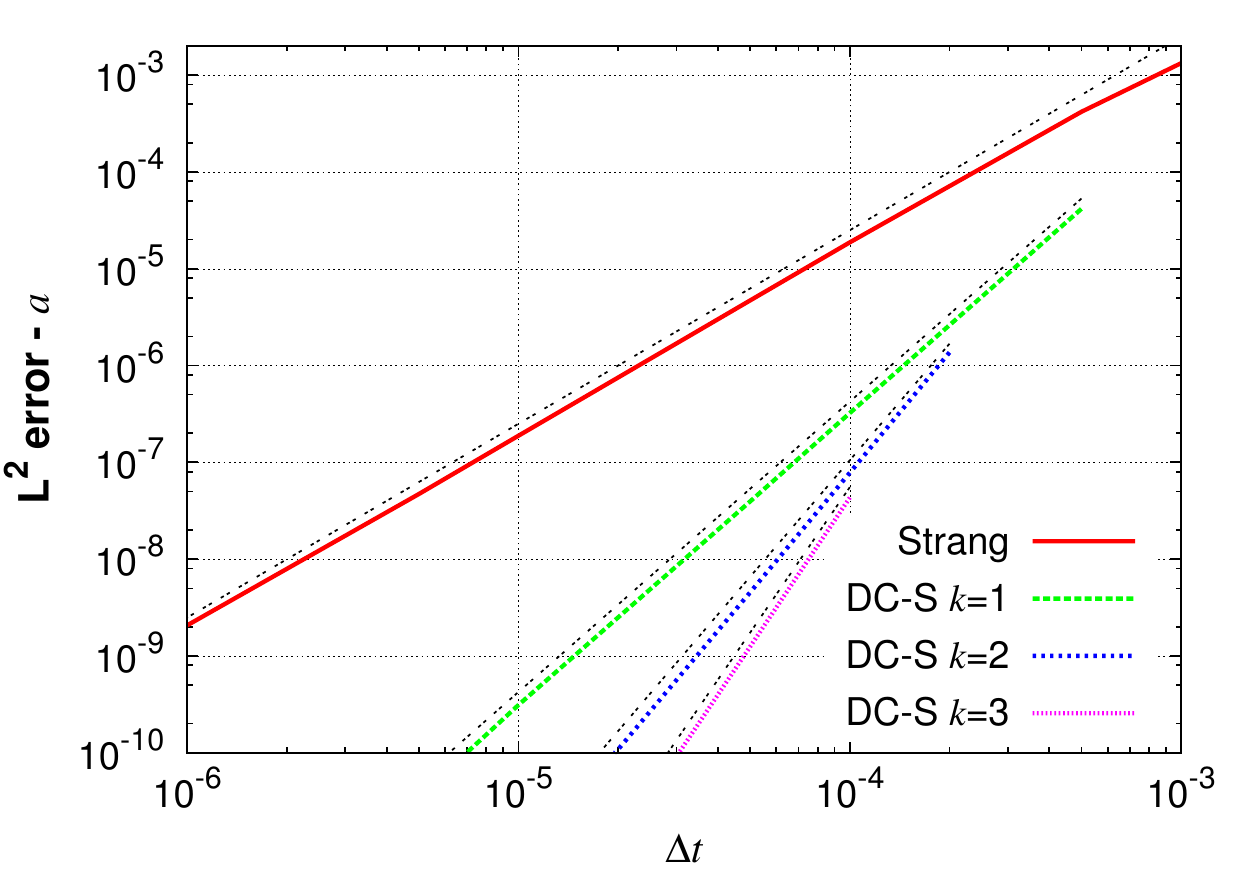}
\end{center}
\caption{Local (left) and global (right) $L^2$--errors
for the DC--S scheme with Strang splitting.
Dashed lines of slopes 3 to 6 (left), and 2 to 5 (right)
are also depicted.}
\label{fig:err_str_nx1001}
\end{figure}

Figure \ref{fig:err_lie_nx1001}
illustrates local and global $L^2$--errors for various
time steps $\step$
with respect to the reference solution,
using the
Lie scheme as the splitting solver into the DC--S scheme.
The Lie solution
in Figure \ref{fig:err_lie_nx1001}
corresponds to the splitting approximation with splitting time step
$\step_{\splt} = \step$.
Local errors are evaluated at $t_0+\step$,
starting in all cases from the reference solution at
$t_0=0.5$.
Global errors are evaluated at final time, $t=1$,
after integrating in time with a constant time step $\step$.
All  $L^2$--errors are scaled by the maximum norm of
variable $a(t,x)$
at the corresponding time.
For this particular problem
the highest order
$p=2s-1=5$
is attained and
local errors behave like
$\Or(\step^{\min[6,2+k]})$
according to
Theorem \ref{teo:teo1},
with $\widehat{p}=1$ for the Lie solver.
Notice that very fine tolerances do not
need to be considered
for the solvers used within the splitting method.
For relatively large time steps we observe the
well--known loss of order related to splitting
schemes on stiff PDEs, as investigated in
\cite{Descombes_LM},
which propagates throughout the DC--S iterations.
Nevertheless, for this particular problem,
this loss of order is
substantially compensated during the time integration,
as observed in the global errors at the final integration time.
Same remarks can be made for the DC--S scheme with Strang
splitting and $\widehat{p}=2$,
as seen in Figure \ref{fig:err_str_nx1001}.
However, the iterative scheme attains
a local order between $q+2=5$ and $p+1=6$,
even though the global errors behave like
$\Or(\step^{\min[5,2+k]})$.

Using the time--stepping procedure established in \S\ref{sec:error_control},
Figure~\ref{fig:err_screen_nx1001}
shows the estimated values of the local error
$\widetilde{err}_k$, according to (\ref{eq:errk_bound2}),
compared to the actual local errors illustrated in
Figures \ref{fig:err_lie_nx1001} and
\ref{fig:err_str_nx1001}.
Recall that $\widetilde{err}_0$ stands for the splitting error.
The estimated maximum time steps at each corrective iteration
$\step_{\max,k}=(\widetilde{\zeta}_k)^{-1}$,
are also indicated.
They represent the maximum time step at which corrections
can be iteratively introduced by the DC--S scheme,
which corresponds to the intersection of local errors
at different iterations in Figures \ref{fig:err_lie_nx1001} and
\ref{fig:err_str_nx1001}.
Notice that the loss of order was not considered in
\S\ref{sec:error_control} to derive these estimates;
however, the computations can be safely performed
because
$\widetilde{err}_k$ overestimates the real local error
for relatively large time steps,
while the $\step_{\max,k}$'s are also smaller than the real ones.
The bound (\ref{eq:res_teo2}) in Corollary \ref{teo:teo2},
on which $\widetilde{err}_k$ is based,
is nevertheless formally guaranteed up to
$k=\min[p-\widehat{p},q-\widehat{p}+1]$, that is,
$k=3$ and $k=2$ for the Lie and Strang splitting, respectively.
\begin{figure}[!htb]
\begin{center}
\includegraphics[width=0.49\textwidth]{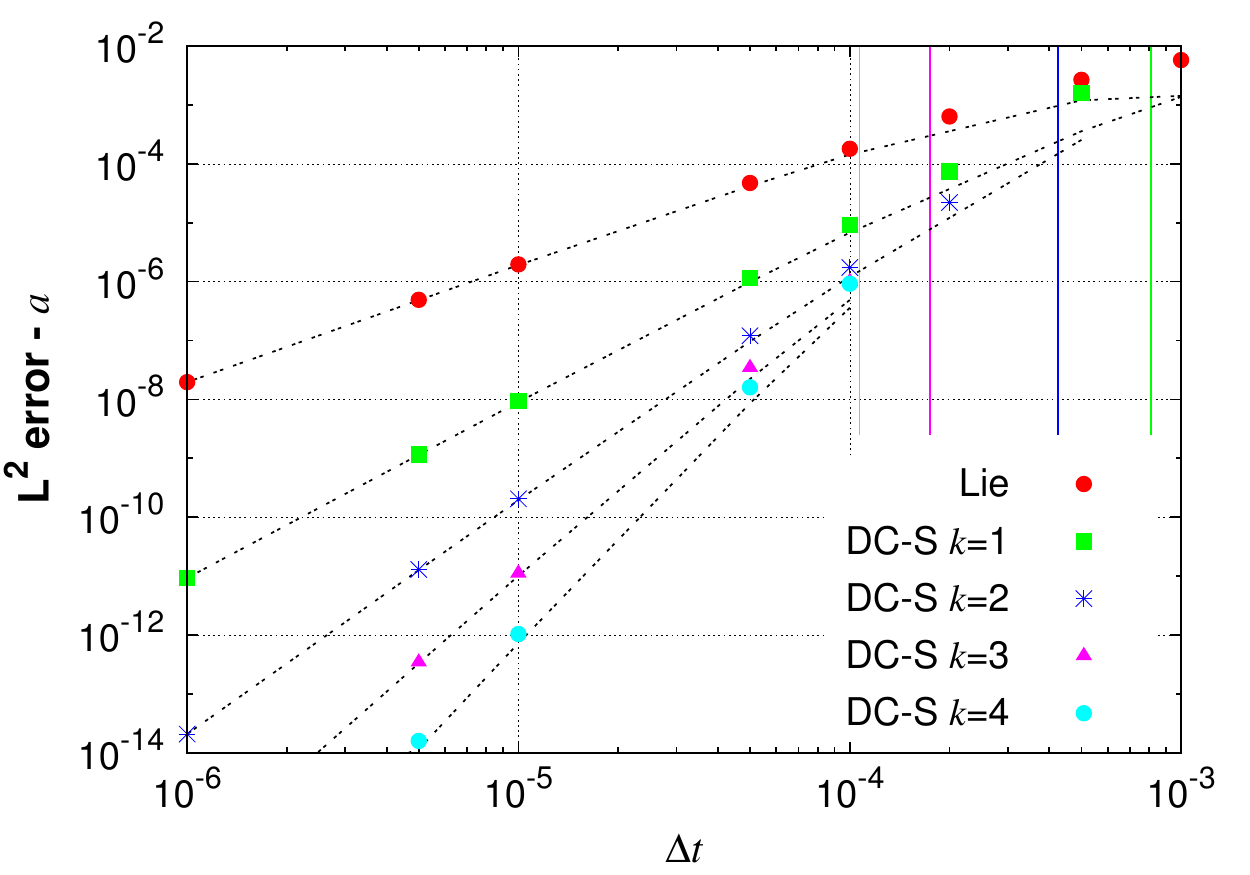}
\includegraphics[width=0.49\textwidth]{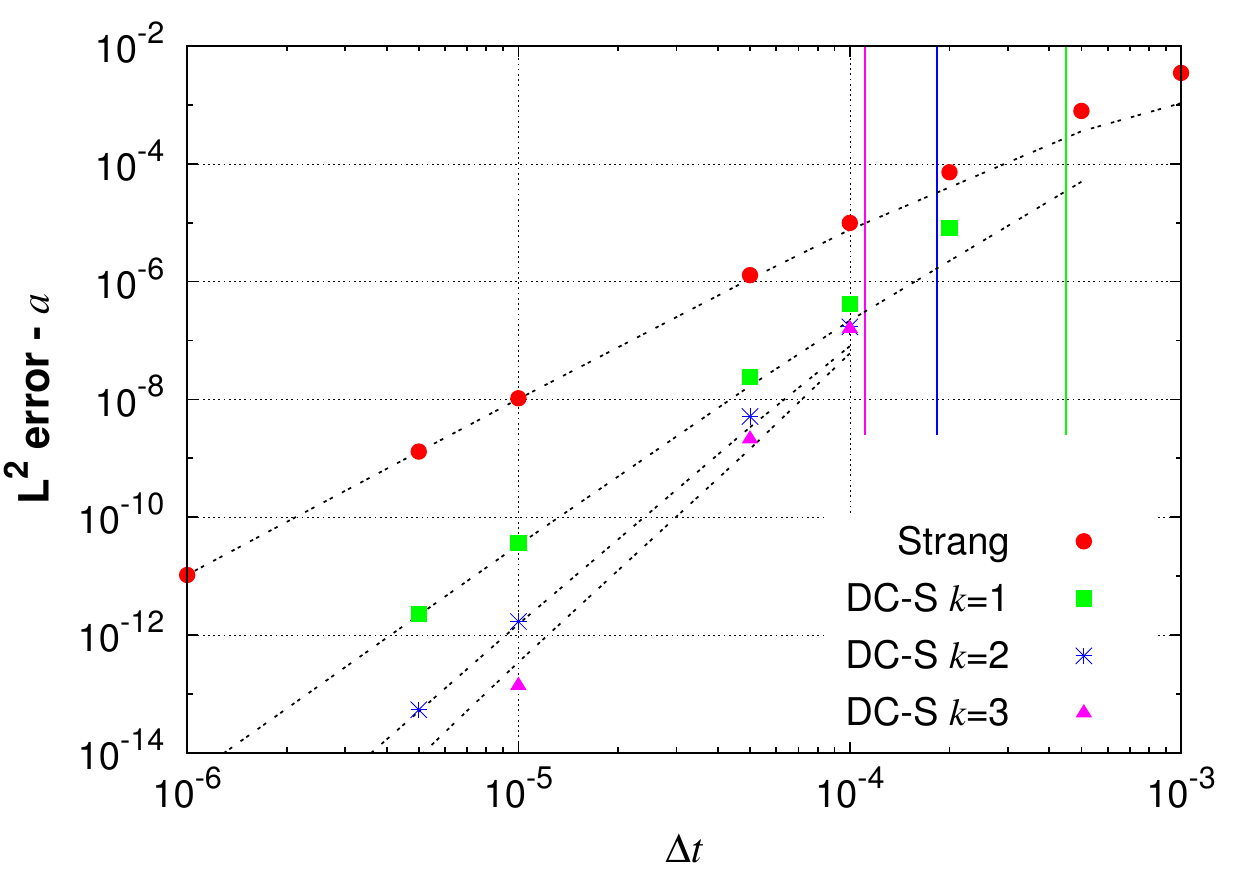}
\end{center}
\caption{Local error estimates $\widetilde{err}_k$ for different
time steps $\step$
for the DC--S scheme with Lie (left) and Strang (right) splitting.
Solid vertical lines stand for the maximum time steps $\step_{\max,k}$.
Dashed lines correspond to local $L^2$--errors from Figures \ref{fig:err_lie_nx1001} and
\ref{fig:err_str_nx1001}.}
\label{fig:err_screen_nx1001}
\end{figure}
\begin{figure}[!htb]
\begin{center}
\includegraphics[width=0.49\textwidth]{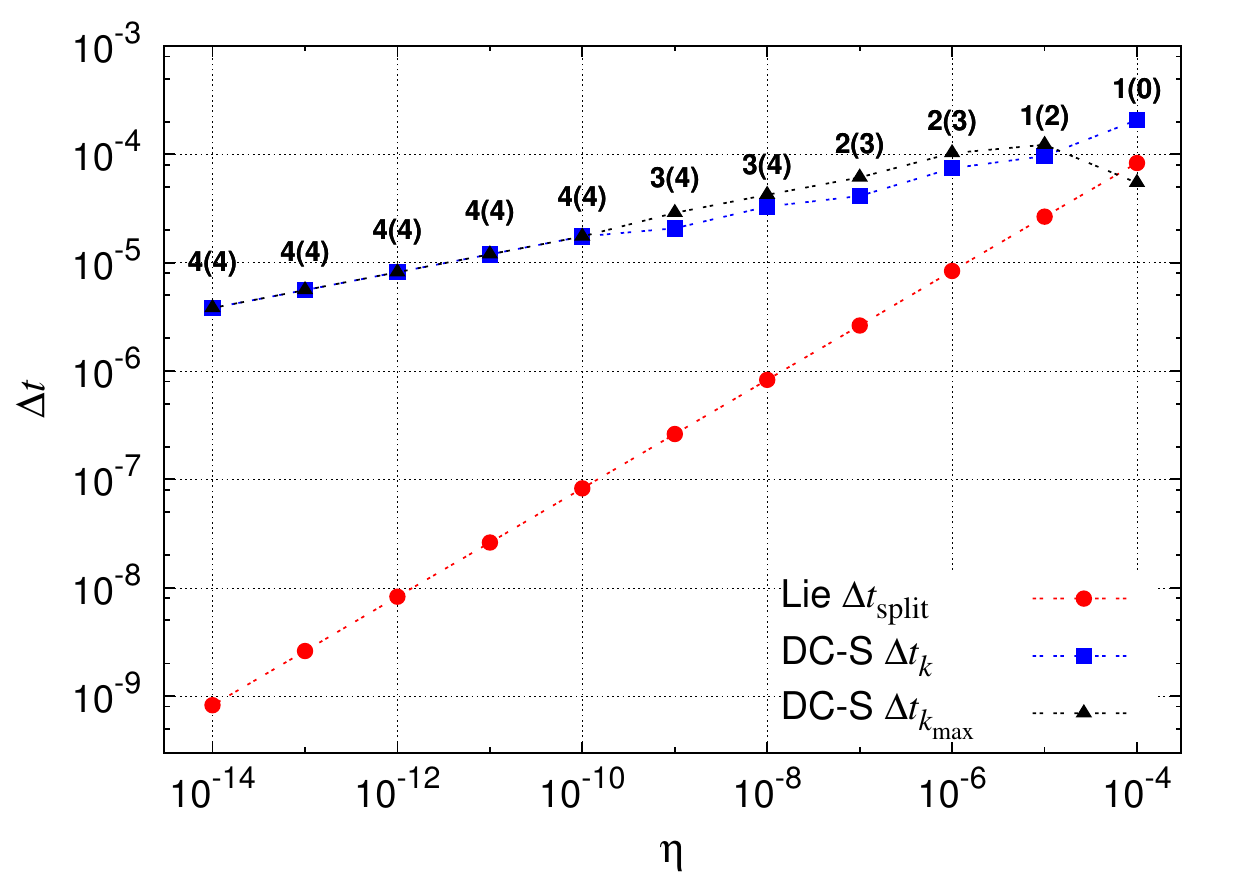}
\includegraphics[width=0.49\textwidth]{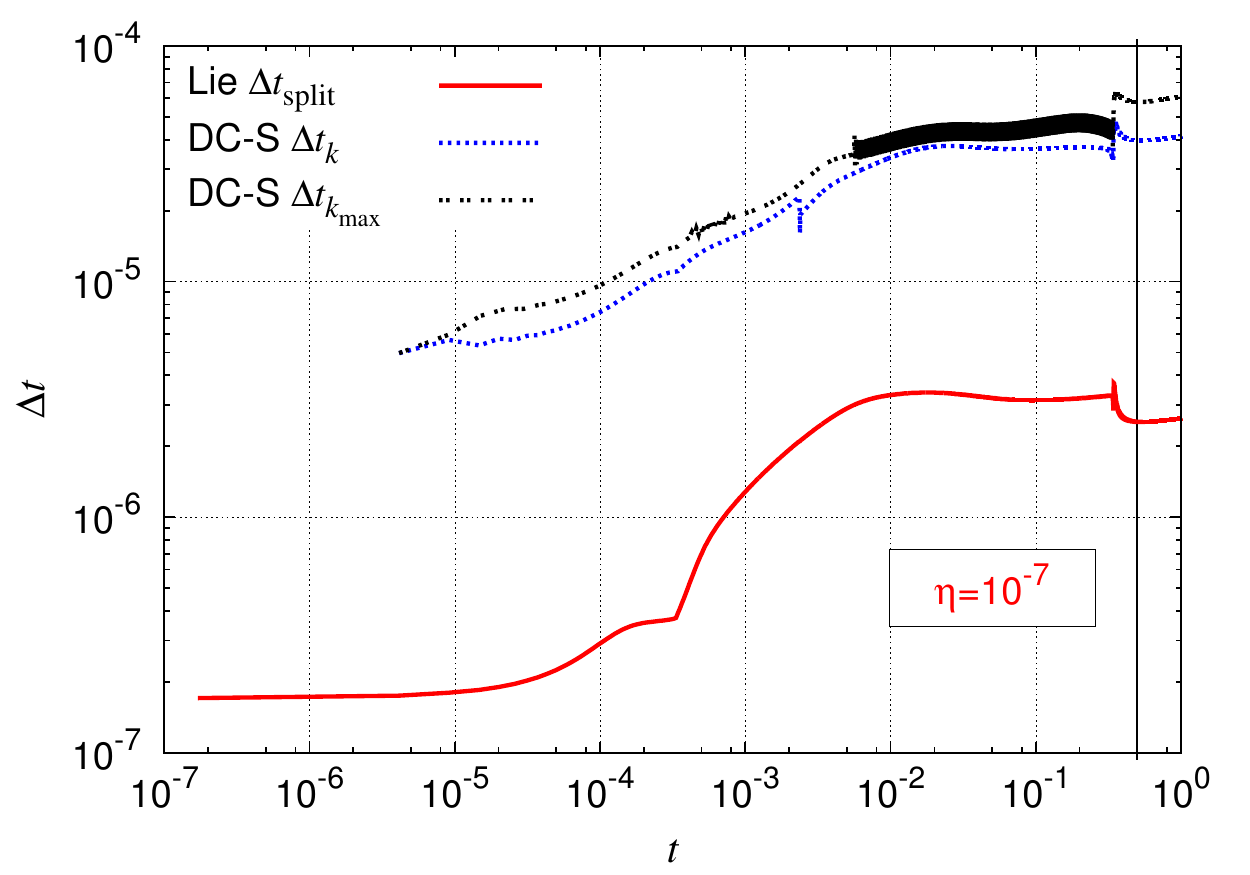}
\end{center}
\caption{Left:
time steps $\step$ at $t=0.5$
for different accuracy tolerances $\eta$
for the Lie and DC--S (Lie) scheme
using $\step_{\new,k}$ (\ref{eq:dtnew})
and $\step_{\new,k_{\max}}$ (\ref{eq:dtnew_kmax})
for time--stepping.
Number of iterations performed are indicated,
the case $\step_{\new,k_{\max}}$ in parenthesis.
Right: time--stepping for $\eta=10^{-7}$;
the solid vertical line corresponds to time $t=0.5$.}
\label{fig:dt_eta_lie_nx1001}
\end{figure}

The advantages of considering a high order approximation in time can
be inferred from Figure~\ref{fig:dt_eta_lie_nx1001}, especially
when fine accuracies are required.
For instance, to attain an accuracy of $\eta=10^{-12}$
the DC--S scheme based on Lie splitting needs a time step of $\step \approx 10^{-5}$,
whereas considering the Lie approximation would require a splitting
time step $1000$ times smaller.
In particular small variations of $\step$ in the DC--S scheme
can greatly improve the accuracy of the time integration at roughly
the same computational cost.
Moreover, in order to guarantee
a splitting solution of accuracy $\eta$, the numerical solution
of the split subproblems must also be performed with the same level
of accuracy \cite{Duarte11_SISC}.
For the splitting solver considered here, the latter involves
$\eta_{\rm Radau5} \leq \eta$
and $\eta_{\rm ROCK4} \leq \eta$,
whereas this is not required for the DC--S scheme.
As previously noted, using
$\step_{\new,k_{\max}}$ (\ref{eq:dtnew_kmax})
for the time--stepping involves slightly larger time steps
because more iterations are performed.
The dynamical time--stepping is illustrated in
Figure~\ref{fig:dt_eta_lie_nx1001} (right), using
$\step_{\new,k}$ (\ref{eq:dtnew})
and $\step_{\new,k_{\max}}$ (\ref{eq:dtnew_kmax})
for the DC--S scheme, and $\step_{\new,\splt}$
(\ref{eq:dtnew_split}) to adapt
the splitting time step for the Lie solver.
For this particular problem, a roughly constant time step is
attained, after some initial transients, consistent with the
quasi self--similar propagation of the wavefronts, as depicted in
Figure~\ref{fig:sol_BZ_nx1001}.
In all cases the time--stepping procedure guarantees numerical approximations
within a user--defined accuracy.

\subsection{Space discretization errors}
We now investigate the spatial discretization errors in the
previous approximations.
We consider both a second and a fourth order centered finite differences
scheme for the Laplacian operator, as well as various resolutions
given by 501, 1001, 2001, 4001, 8001, and 16001 grid points.
The
DC--S solution (with Lie splitting and $k=4$)
computed with the
finest space resolution of 16001 points
is taken as the reference solution.
Figure \ref{fig:space_err} (left) shows the numerical errors for
different discretizations with respect to the reference solution.
All approximations are initialized from the same solution represented
on the 16001--points reference grid;
then, an integration time step is computed on
all grids with the same
DC--S integration scheme
and a time step of $10^{-6}$.
The goal is to illustrate only the space discretization
errors,
as seen in Figure \ref{fig:space_err} (left).

Considering the same results obtained with the Lie and the corresponding
DC--S scheme on a 1001--points grid, previously shown in Figure \ref{fig:err_lie_nx1001},
we compute the numerical errors
with respect to the reference solution on 16001 points.
Figure \ref{fig:space_err} (right) illustrates the errors arising
from spatial and temporal numerical errors after one time step.
According to (\ref{eq:err_uk_disc}) we see that for sufficiently
large time steps the temporal integration error is mainly responsible
for the total approximation error, whereas the accuracy of the method
is limited by the space discretization errors for small time steps.
Considering a high order space discretization scheme increases
the region for which time integration errors control the accuracy
of the approximation.
For a relatively coarse grid of
1001 grid points, it can be seen that an order of
magnitude can be already gained in terms of accuracy by increasing the order
of the space discretization scheme, as seen in Figure \ref{fig:space_err}.
Greater improvements are expected for finer spatial resolutions
as inferred from Figure \ref{fig:space_err} (left).
\begin{figure}[!htb]
\begin{center}
\includegraphics[width=0.49\textwidth]{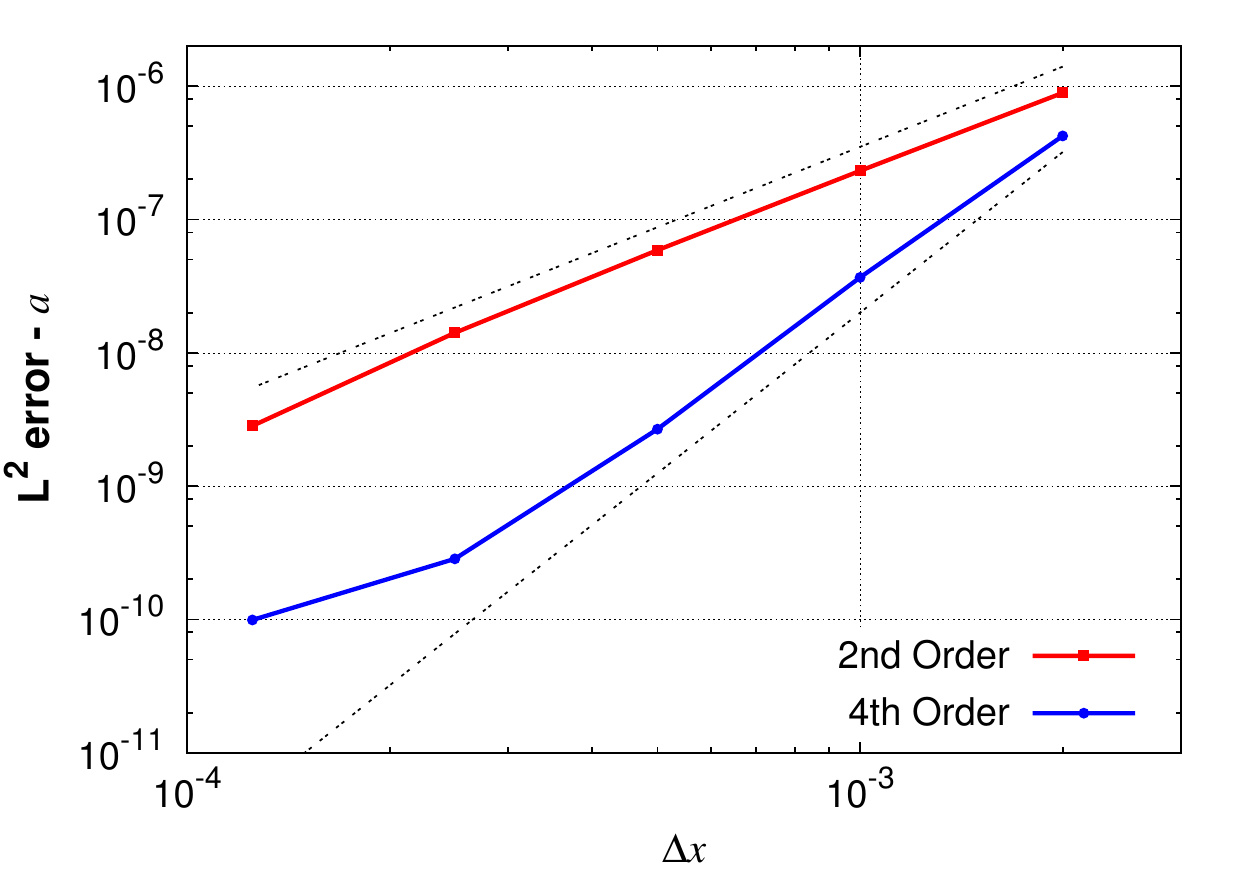}
\includegraphics[width=0.49\textwidth]{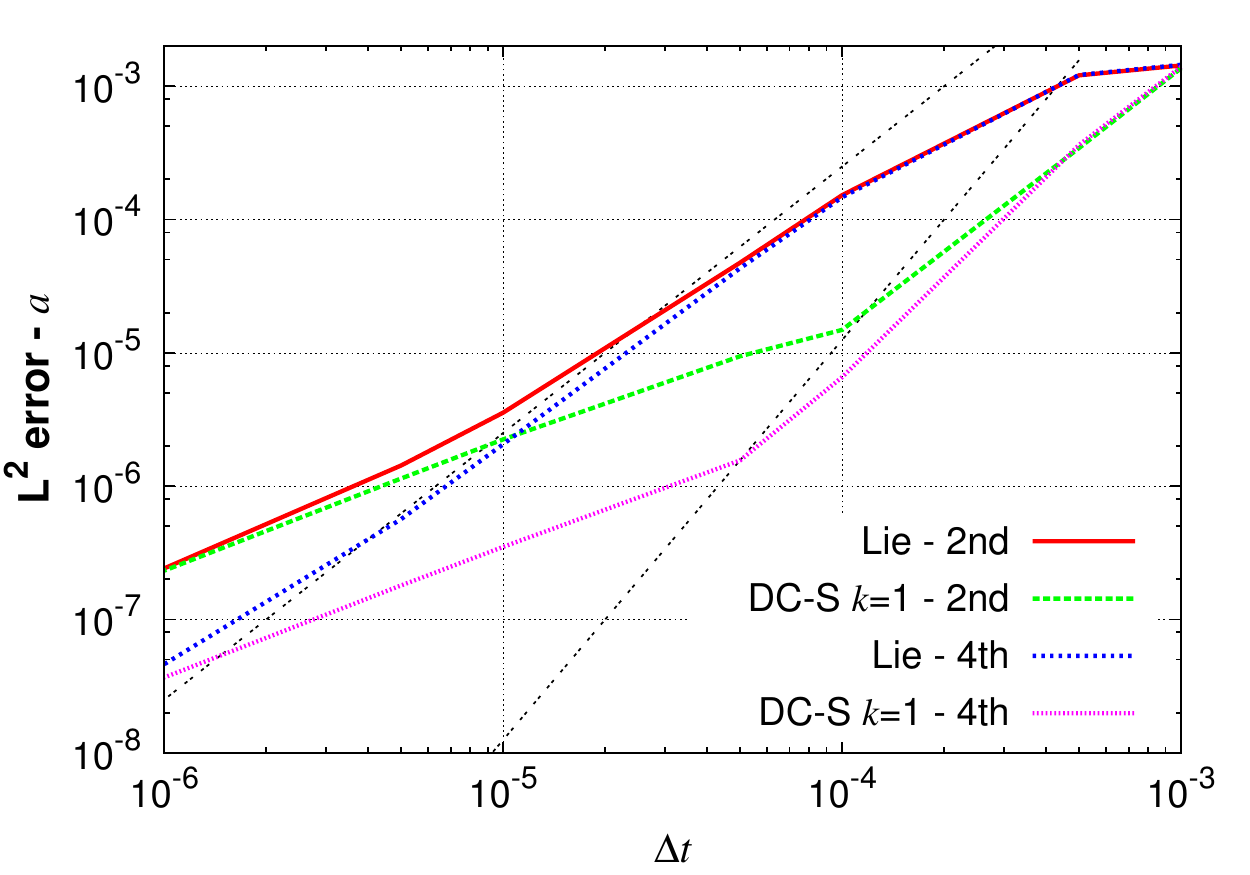}
\end{center}
\caption{Left: space discretization errors after one time step with the same time integration
scheme.
Right: local $L^2$--errors for the DC--S scheme with Lie splitting
on a 1001--points grid, accounting for both time and
space numerical errors.
Dashed lines of slopes 2 and 4 (left), and 2 and 3 (right)
are also depicted.
Space discretizations of second and fourth order are considered.
}
\label{fig:space_err}
\end{figure}
\begin{figure}[!htb]
\begin{center}
\includegraphics[width=0.49\textwidth]{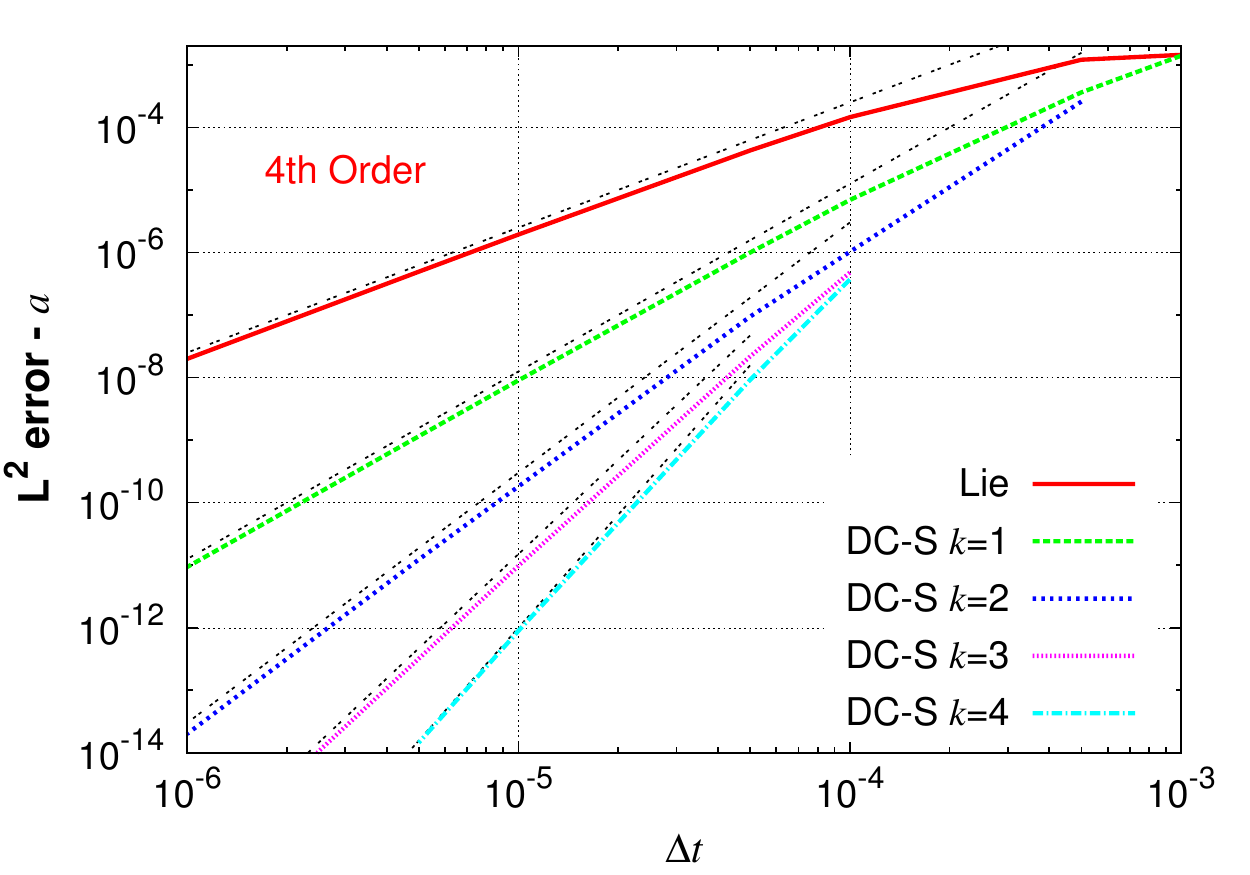}
\includegraphics[width=0.49\textwidth]{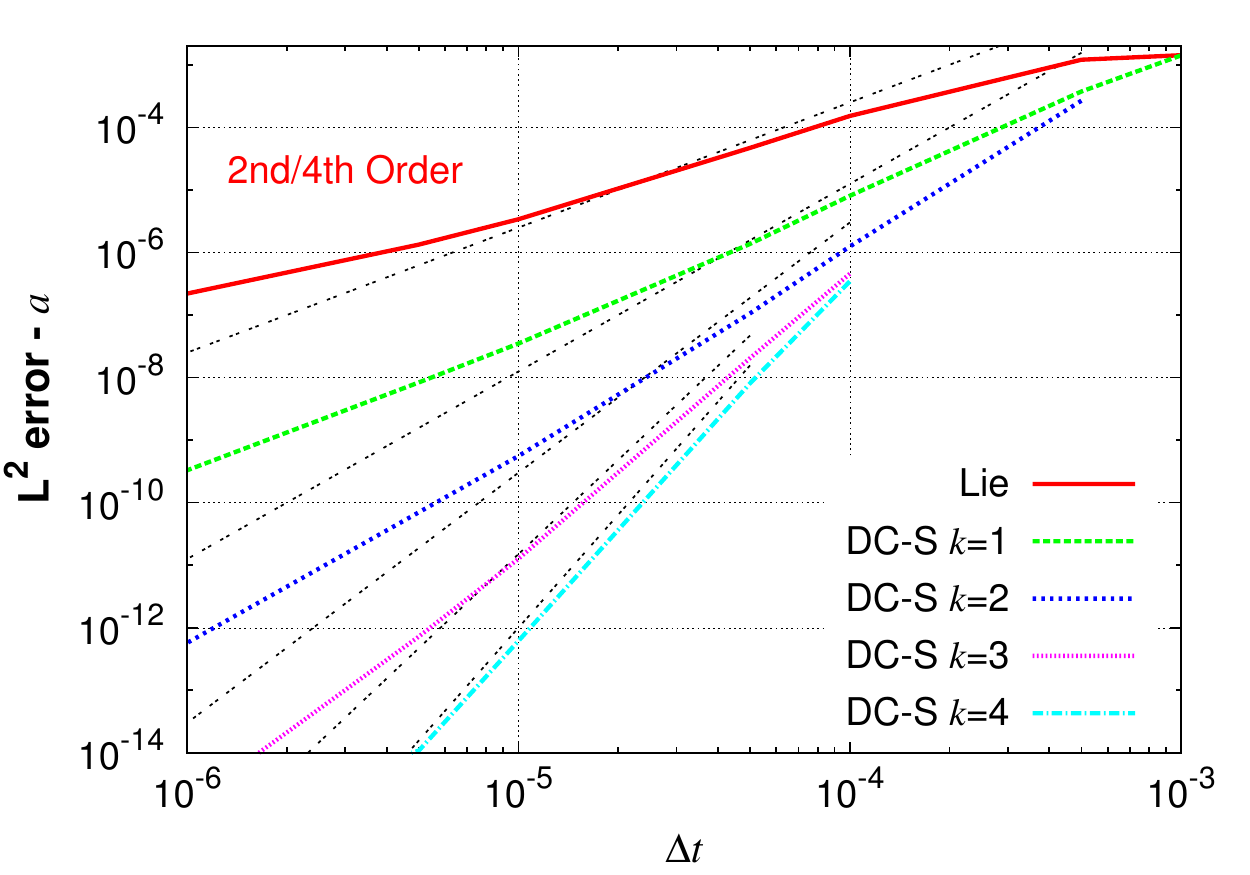}
\end{center}
\caption{Local $L^2$--errors
for the DC--S scheme with Lie splitting with fourth order
(left) and second/fourth order (right)
space discretizations.
Dashed lines of slopes 3 to 6 are also depicted.}
\label{fig:space_err_4th}
\end{figure}

Coming back to a spatial resolution of 1001 points, Figure \ref{fig:space_err_4th}
(left) displays numerical errors related to the time integration,
as in Figure \ref{fig:err_lie_nx1001} (left), this time with the fourth order
space discretization. As before the reference solution is obtained
with Radau5 and $\eta_{\rm Radau5} = 10^{-14}$,
considering now the fourth order
space discretization scheme.
The same previous conclusions apply
as for the results in Figure \ref{fig:err_lie_nx1001},
now with high order approximations in both time and space.
Similar behaviors are observed for the Strang case.
Nevertheless,
a high order scheme in space can be more
time and memory consuming, since it usually requires larger
stencils and thus more computational work.
One alternative,
previously discussed in \S\ref{subsubsec:high_order_space},
uses high order space discretizations
for the quadrature formulas, which in practice involve
matrix--vector multiplications; while the time integrators
that generate the iterative corrections are performed
with low order discretization schemes in both time and space.
In Figure \ref{fig:space_err_4th} (right) we evaluate
the temporal errors for this configuration, using
the same reference solution as in Figure \ref{fig:space_err_4th} (left).
It can be seen that the space discretization errors are
progressively eliminated until we obtain the same results
observed in Figure \ref{fig:space_err_4th} (left).
That is, the hybrid implementation with second/fourth order in space
converges to the high order quadrature formula evaluated
with a high order space discretization.
However,
the time stepping strategy with error control
established
in \S\ref{sec:error_control}
is no longer valid
since a decomposition of time and space discretization errors
such as (\ref{eq:err_uk_disc})
is no longer possible.

\section{Concluding remarks}\label{sec:conclusion}
We have introduced a new numerical scheme that achieves high order
approximations in time for the solution of time dependent stiff PDEs.
The method exploits the advantages of operator splitting
techniques to handle the stiffness associated with different
processes modeled by the PDEs, while high order in time
is attained through an iterative procedure based on
a standard deferred correction technique.
In this way the initially low order approximation computed by
a tailored splitting solver is iteratively corrected
to obtain a high order approximation based on a quadrature
formula.
The latter is based on Radau collocation nodes according
to the RadauIIA, $s$--stage quadrature formulas.
Moreover, the splitting solver uses dedicated methods
in terms of numerical stability and independent time--stepping features
within every splitting time step
to separately advance each subproblem
originating from the PDE.
Consequently, no time step restriction needs to be observed
in the deferred correction splitting technique to guarantee
the numerical stability of the numerical integration.
The traditionally low algorithmic complexity and efficiency
of a splitting approach are therefore preserved, while high order
approximations in time are achieved by including extra function evaluations at
the collocation nodes and by performing matrix--vector multiplications
according to the quadrature formulas.

A mathematical analysis of the method was also conducted
in a general finite dimensional space with standard assumptions on the
PDEs and the splitting approximations.
In this context
it was proved that the local error of the iterative method behaves
like $\step^{\min[p+1,q+2,\widehat{p}+k+1]}$,
$k=1,2,\ldots$,
where $p$, $q$, and $\widehat{p}$ stand for the global and stage
orders of the quadrature formulas and the global order of the splitting
approximation, respectively.
A maximum time step was also formally identified beyond which the
iterative scheme yields no additional correction.
Based on these theoretical results, a time--stepping strategy
was derived to monitor the approximation errors
related to the time integration.
Both the definition of the maximum time step
as well as the time--stepping procedure
remain valid regardless of the integration scheme
used to approximate the solution at the collocation nodes.
Numerical results confirmed the theoretical findings in terms
of time integration errors and the orders attained.
In particular, given a stiff PDE discretized on a given grid,
the time--stepping technique yields numerical time integrations
within a user--defined accuracy tolerance.
This error control remains effective for the overall accuracy
of the numerical approximations for a sufficiently fine space resolution
and/or when high order space discretization schemes are considered.

Finally, a hybrid approach wherein low order spatial discretizations
are used to advance the correction equation while fully coupled high
order spatial discretizations are used to evaluate the quadrature
formula was introduced.  The numerical tests performed herein
demonstrate that this approach converged, for the particular problem
considered, to the same solution obtained using the high order spatial
discretization throughout.  The advantage to the hybrid approach is
that the correction equation is solved using less computationally
expensive spatial operators, although the theoretical results
concerning adaptive time--stepping no longer apply.


\appendix

\section{Relation with the parareal algorithm}\label{sec:parareal}
Considering problem (\ref{eq:gen_prob_disc}) over the time domain
$[t_0,t_n]$, it can be decomposed
into $n$ subdomains $[t_{i-1},t_{i}[$,
$i=1,\ldots,n$, and
$\step_i:=t_{i}-t_{i-1}$.
The parareal algorithm \cite{Lions01}
is based on two propagation operators:
$\Gpara^{t} \vec u_0$
and $\Fpara^{t}\vec u_0$,
that provide, respectively, a coarse and a more accurate
(fine) approximation
to the solution of problem (\ref{eq:gen_prob_disc}).
The algorithm
starts with an initial approximation
$(\widetilde{\vec u}^0_i)_{i=1,\ldots,n}$,
given by the sequential computation:
\begin{equation}\label{con_ini}
{\widetilde{\vec u}}_i^0=\Gpara^{\step_{i}}{\widetilde{\vec u}}_{i-1}^0, \quad
i=1,\ldots,n.
\end{equation}
It then performs the correction iterations:
\begin{equation}\label{para}
\widetilde{\vec u}^{k+1}_i=
 \Fpara^{\step_{i}} {\widetilde{\vec u}}^{k}_{i-1}
+\Gpara^{\step_{i}} {\widetilde{\vec u}}^{k+1}_{i-1}
-\Gpara^{\step_{i}} {\widetilde{\vec u}}^{k}_{i-1},
\quad
i=1,\ldots,n.
\end{equation}
The coarse approximations, which are computed sequentially,
are thus iteratively corrected in order to get
a more accurate solution based on the fine solver,
performed in parallel.

The coarse propagator can be thus defined as a low order
operator splitting technique as considered in
\cite{Duarte11_parareal}.
Moreover, if the $n$ time intervals are defined according
to the collocation nodes of an IRK quadrature formula, then
the fine propagator could be defined following
(\ref{eq2:def_uhat}).
The iterative DC--S scheme
(\ref{eq:ini_split})--(\ref{eq2:def_utildeksplit})
can be thus formally seen as a particular implementation of the parareal
algorithm (\ref{con_ini})--(\ref{para}).
The parareal algorithm can be actually recast
as a deferred correction scheme \cite{Gander07}.
The combination of time parallelism and deferred correction techniques was
recently investigated in \cite{Minion2010,Emmett2012}.
However, in the DC--S scheme
the deferred corrections
are performed over a single time step instead of the
entire time domain of interest.
The latter involves a departure from the time parallel
capabilities of the parareal algorithm,
as it was originally conceived.
Furthermore, the quadrature formula in the DC--S scheme
that considers the approximations
$(\widetilde{\vec u}^k_i)_{i=1,\ldots,s}$ to compute
$\widehat{\vec u}^k_i$ at each node
cannot be
viewed as the standard fine propagator in the parareal framework
which would rather perform a numerical time integration from $\widetilde{\vec u}^k_{i-1}$.
The latter is especially relevant
when one studies the method
since a common approach in the analysis of the parareal algorithm considers
the fine propagator as the semiflow corresponding to the exact solution (see, e.g.,
\cite{Lions01,Bal03,Gander07,Hairer08}).

\section{Relation with SDC methods}\label{sec:SDC}
The iterative
SDC method introduced in \cite{DuttSDC2000}
to approximate the solution of problem (\ref{eq:gen_prob_disc})
can be written in general as \cite{Minion2003}
\begin{equation}\label{eq:uSDC_k}
\widetilde{\vec u}^{k+1}_i = \widetilde{\vec u}^{k+1}_{i-1} +
\int_{t_{i-1}}^{t_i} \left[ \vec F(\widetilde{\vec u}^{k+1}(\tau))
- \vec F(\widetilde{\vec u}^k(\tau))\right] \, \der \tau
+ S_{t_{i-1}}^{t_i} (\widetilde{\vec U}^k),
\quad
i=1,\ldots,n,
\end{equation}
where $\step=t_n-t_0$, and
\begin{equation}\label{eq:SDC_S}
S_{t_{i-1}}^{t_{i}} ({\vec U})
\approx
 \int_{t_{i-1}}^{t_{i}} \vec F({\vec u}(\tau))\, \der \tau,
\end{equation}
is a spectral integration operator
defined
by means of a quadrature formula that evaluates
$\vec F(\vec u(t))$ at
the $n$ or $n+1$ (if time $t_0$ is included)
collocation nodes.
Depending on the stiffness of problem (\ref{eq:gen_prob})
(or (\ref{eq:gen_prob_disc}))
either an explicit or implicit Euler approximation
to the remaining integrals in (\ref{eq:uSDC_k})
is considered, that is,
\begin{equation*}\label{eq:uSDC_k_exp}
\widetilde{\vec u}^{k+1}_{i} = \widetilde{\vec u}^{k+1}_{i-1} +
\step_i
\left[ \vec F(\widetilde{\vec u}^{k+1}_{i-1}(\tau))
- \vec F(\widetilde{\vec u}^k_{i-1}(\tau))\right]
+ S_{t_{i-1}}^{t_{i}} (\widetilde{\vec U}^k),
\end{equation*}
or
\begin{equation*}\label{eq:uSDC_k_imp}
\widetilde{\vec u}^{k+1}_{i} = \widetilde{\vec u}^{k+1}_{i-1} +
\step_i
\left[ \vec F(\widetilde{\vec u}^{k+1}_{i}(\tau))
- \vec F(\widetilde{\vec u}^k_{i}(\tau))\right]
+ S_{t_{i-1}}^{t_{i}} (\widetilde{\vec U}^k),
\end{equation*}
or a combination of both depending on the nature
of each $\vec F_i(u)$ into $\vec F(u) = \vec F_1(u) + \vec F_2(u) + \ldots$
\cite{Minion2003}.
The $0$--th iteration is computed by a low order time integration scheme.
These approximations are thus iteratively
corrected to obtain a high order quadrature rule
solution:
$\widetilde{\vec u}^{k+1}_{n} \approx
\widetilde{\vec u}^{k+1}_{n-1} + S_{t_{n-1}}^{t_{n}} (\widetilde{\vec U}^k)$
following (\ref{eq:uSDC_k}).

Introducing the following approximation into (\ref{eq:uSDC_k}),
\begin{equation*}
\widetilde{\vec u}^k_{i-1} +
\int_{t_{i-1}}^{t_{i}} \vec F(\widetilde{\vec u}^{k}(\tau))
\, \der \tau \approx
\OS^{(c_i-c_{i-1})\step} \widetilde{\vec u}^{k}_{i-1},
\end{equation*}
with $\step_1=c_1\step$ and $\step_i=(c_i-c_{i-1})\step$,
$i=2,\ldots,n=s$,
one can notice that
(\ref{eq2:def_utildeksplit}) and (\ref{eq:uSDC_k})
become equivalent,
as long as the spectral operator $S_{t_{i-1}}^{t_{i}} (\vec U)$
with its corresponding quadrature nodes are defined
based on an IRK scheme relying on a collocation method,
that is,
$S_{t_{i-1}}^{t_{i}} (\vec U)
\equiv
I_{t_{i-1}}^{t_{i}} (\vec U)$.
The spectral operator (\ref{eq:SDC_S})
was built
in \cite{DuttSDC2000}
based on the Lagrange interpolant
that evaluates $\vec F(\vec u(t))$
at the Gauss--Legendre nodes,
which is equivalent to considering a Gauss quadrature
formula for IRK schemes.
The latter was later enhanced in \cite{Minion2003}
by employing Lobatto quadrature formulas.
In this case
the spectral operator is equivalent to that based
on the LobattoIIIA--IRK scheme.
These analogies between spectral operators in SDC
schemes and quadrature formulas for IRK
methods were investigated, for instance, in \cite{Hagstrom2006}.
The DC--S scheme (\ref{eq2:def_utildeksplit}) can be thus
recast as a variation of the SDC method in which numerical
time integrations are performed within each time
subinterval instead of
approximating integrals into (\ref{eq:uSDC_k}).
The latter, however, involves important changes to the iterative procedure
in (\ref{eq:uSDC_k}) together with its practical implementation.

%
\bibliographystyle{plain}
\bibliography{biblio_DCS}

\end{document}